\documentclass[11pt]{article}

\usepackage[T1]{fontenc}
\usepackage{epsfig}
\usepackage{amsmath, amsthm}
\usepackage{amsfonts,amssymb}
\usepackage{mathtools}
\usepackage[applemac]{inputenc}		
\usepackage{hyperref}
\usepackage{comment}

\hypersetup{colorlinks=true,pdftitle=""}

\newtheorem{conj}{Conjecture}[section]
\newtheorem{thm}[conj]{Theorem}

\newtheorem{rem}[conj]{Remark}
\newtheorem{lem}[conj]{Lemma}
\newtheorem{prop}[conj]{Proposition}
\newtheorem{coro}[conj]{Corollary}
\newtheorem{ques}[conj]{Question}

\newtheorem{defn}[conj]{Definition}
\newtheorem{cor}[conj]{Corollary}

\newtheorem*{theorem*}{Theorem}

\newcommand{\f}{\varphi}

\newcommand{\calA}{\mathcal{A}}
\newcommand{\calB}{\mathcal{B}}

\newcommand{\N}{\mathbb{N}}

\renewcommand{\L}{\mathcal{L}}

\newcommand{\bc}{\begin{center}}
\newcommand{\ec}{\end{center}}
\newcommand{\bt}{\begin{tabular}}
\newcommand{\et}{\end{tabular}} 
\newcommand{\bea}{\begin{eqnarray}}
\newcommand{\eea}{\end{eqnarray}}
\newcommand{\bean}{\begin{eqnarray*}}
\newcommand{\eean}{\end{eqnarray*}}

\newcommand{\ba}{\begin{array}}
\newcommand{\ea}{\end{array}}

\def\be{\begin{eqnarray}}
\def\ee{\end{eqnarray}}
\def\ben{\begin{eqnarray*}}
\def\een{\end{eqnarray*}}

\newcommand{\ra} {\rightarrow}

\newcommand{\RL}{{\mathbb R}}

\newcommand{\calM}{\mbox{${\cal M}$}}

\newcommand{\Rpl}{{\mathbb R}_{+}}

\newcommand{\dou}{\partial}

\newcommand{\eps}{\epsilon}

\def\elabel#1{\label{e:#1}}

\def\sq{$\Box$}

\def\qed{\ifmmode\sq\else{\unskip\nobreak\hfil
\penalty50\hskip1em\null\nobreak\hfil\sq
\parfillskip=0pt\finalhyphendemerits=0\endgraf}\fi\par\medbreak}

\newsavebox{\junk}
\savebox{\junk}[1.6mm]{\hbox{$|\!|\!|$}}

\def\det{{\mathop{\rm det}}}

\newcommand{\one}{\hbox{\rm\large\textbf{1}}}

\def\til={{\widetilde =}}

 \def\eq#1/{(\ref{#1})}

\def\eq#1/{(\ref{e:#1})}

\newcommand{\beqn}[1]{\notes{#1}%
\begin{eqnarray} \elabel{#1}}

\newcommand{\eeqn}{\end{eqnarray} }

\newcommand{\beq}[1]{\notes{#1}%
\begin{equation}\elabel{#1}}

\newcommand{\eeq}{\end{equation}} 

\def\bdes{\begin{description}}
\def\edes{\end{description}}

\def\notes#1{}

\def\L{L_\infty}

\newcommand{\setS}{S}
\newcommand{\setT}{T}

\newcommand{\collS}{\mathcal{C}}

\newcommand{\qedwhite}{\hfill \ensuremath{\Box}}

\newcommand{\E}{\mathcal{E}}

\def\phi{\varphi}

\def\L{{\cal L}}

\def\bee{\begin{eqnarray*}}
\def\ene{\end{eqnarray*}}

\renewcommand{\eps}{\varepsilon}

\newcommand{\calK}{\mathcal{K}}

\title{On the volume of the Minkowski sum of zonoids}
\author{Matthieu Fradelizi\thanks{supported in part by the Agence Nationale de la Recherche, project GeMeCoD (ANR 2011 BS01 007 01).}, 
Mokshay Madiman\thanks{supported in part by the U.S. National Science Foundation through grants DMS-1409504.}, 
Mathieu Meyer
and Artem Zvavitch\thanks{supported in part by the U.S. National Science
Foundation Grant DMS-1101636, Simons Foundation, the CNRS and the B\'ezout Labex.}}
\date{\today}

\begin{document}

\maketitle

\begin{abstract} 
We explore some inequalities in convex geometry restricted to the class of zonoids. We show the equivalence, in the class of zonoids, between a local Alexandrov-Fenchel inequality, a local Loomis-Whitney inequality, the log-submodularity of volume, and the Dembo-Cover-Thomas conjecture on the monotonicity of the ratio of volume to the surface area. In addition to these equivalences, we confirm these conjectures in $\RL^3$ and we establish an improved inequality in $\RL^2$. Along the way, we give a negative answer to a question of Adam Marcus regarding the roots of the Steiner polynomial of zonoids. We also investigate analogous questions in the $L_p$-Brunn-Minkowski theory, and in particular, we confirm all of the above conjectures in the case $p=2$, in any dimension.
\end{abstract}


\tableofcontents

\section{Introduction}


The deep parallels between inequalities in Information Theory and Convex Geometry have been explored intensively. The recognition of these parallels goes back to at least 1984, when Costa and Cover \cite{CC84} observed the formal resemblance between the Brunn-Minkowski inequality in Convex Geometry and the entropy power inequality in Information Theory.
The Brunn-Minkowski inequality (see, e.g., \cite{Gar02, Sch14:book})
states that for all compact sets $A,B$ in $\RL^n$,
$|A+B|^{1/n} \geq |A|^{1/n} + |B|^{1/n}$, where we write $|A|$ for the  volume of $A$.
The entropy power inequality (see, e.g., \cite{CT91:book}) states that for any pair of independent random vectors $X, Y$ in $\RL^n$,
$N(X+Y) \geq N(X) + N(Y)$,
where
$N(X) = e^{\frac{2h(X)}{n}}$
denotes the entropy power of $X$, and  
$h(X) = -\int f(x) \log f(x) dx$ is the entropy of $X$ if $X$ has density $f$ (and $h(X)=- \infty$ if not). Despite the apparent distinctness of the settings, a compelling case can be made (see, e.g., \cite{DCT91, SV00, Gar02, WM14, MMX17:0}) that these inequalities are deeply connected -- in particular, it is now well understood that the functional $A\mapsto |A|^{1/n}$ in the geometry of compact subsets of $\RL^n$, and the functional $f_X\mapsto N(X)$ in probability are analogous in many (but not all) ways. 
In the last decade, several further developments have been made that link 
Information Theory to the Brunn-Minkowski theory, including entropy analogues of Blaschke-Santal\'o's inequality \cite{LYZ04},
 reverse Brunn-Minkowski's inequality \cite{BM11:cras, BM12:jfa},  Rogers-Shephard's inequality \cite{BM13, MK18} and Busemann's inequality \cite{BNT16}.
Indeed, volume inequalities and entropy inequalities (and also certain small ball inequalities \cite{MMX17:1}) can be unified using
the framework of R\'enyi entropies; this framework and the relevant literature is surveyed in \cite{MMX17:0}. 
 
The analogies between Information Theory and Convex Geometry are not so direct when one moves to second-order functionals. In particular, convex geometry analogues of Fisher information inequalities hold sometimes but not always  \cite{FGM03, AFO14, FM14}.
%
Motivated by these analogies, Dembo, Cover and Thomas \cite{DCT91} proposed the inequality, called concavity of the ratio of volume to the surface area,
\begin{equation}\label{conj:strong}
\frac{|A+B|}{|\partial (A+B)|}\geq \frac{|A|}{|\partial A|}+\frac{|B|}{|\partial B|}
\end{equation}
as a natural analogue dual to the Fisher information inequality. It was noticed already in \cite{DCT91} (see also \cite{GHP02}) that \eqref{conj:strong} is true when $A$ (or $B$) is an Euclidean Ball, moreover,  in \cite{FGM03}, it was proved that this last inequality holds for any convex bodies $A$ and $B$ in $\RL^2$. It is interesting to note that even a weaker conjecture, called monotonicity of the ratio of volume to the surface area,
\begin{equation}\label{conj:weak}
\frac{|A+B|}{|\partial (A+B)|}\geq \frac{|A|}{|\partial A|}
\end{equation}
is not trivial.  The case of $B$ being a segment already gives another natural conjecture
\begin{equation}\label{conj:surfproj}
\frac{|\partial(P_{u^\perp}A)|}{|P_{u^\perp}A|}\le \frac{|\partial A|}{|A|},
\end{equation}
where $P_{u^\perp} A$ denotes the orthogonal projection  $A$ onto the hyperplane with normal vector $u\in S^{n-1}$. We  explain  in details the relationships between the above conjectures as parts of Theorems  \ref{thm:logsubmod-equiv} and \ref{th:linear} below.
It was proved in \cite{GHP02} (see also Lemma \ref{lem:GHPnew}  for a simple proof) that (\ref{conj:surfproj}) holds with with a multiplicative constant:
\begin{equation}\label{eq:surfproj}
\frac{|\partial(P_{u^\perp}A)|}{|P_{u^\perp}A|}\le \frac{2(n-1)}{n} \frac{|\partial A|}{|A|},
\end{equation}
Moreover, it  was shown in  \cite{FGM03} that  the constant is sharp. Thus, there are counterexamples to (\ref{conj:surfproj})  in $\RL^n$ for any $n\ge3$ and inequalities (\ref{conj:strong}) and (\ref{conj:weak}) are not true in general for the whole class of convex bodies in $\RL^n$, $n\ge 3$.
Conjecture (\ref{conj:strong}) is connected with the following conjectured inequality for the volume of projections:
\begin{equation}\label{constrong}
\frac{|A+B|}{|P_{u^\bot} (A+B)|_{n-1}}\geq \frac{|A|}{|P_{u^\bot}A|_{n-1}}+\frac{|B|}{|P_{u^\bot} B|_{n-1}}.
\end{equation}
Bonnesen proved in \cite{Bon29:book} (see  \cite{Sch14:book} equation (7.196)) that, for any convex bodies $A$, $B$ in $\RL^n$,
\begin{equation}\label{eq:bonnesen}
|A+B|\ge \left(|P_{u^\bot} A|_{n-1}^\frac{1}{n-1}+|P_{u^\bot} B|_{n-1}^\frac{1}{n-1}\right)^{n-1} \left(\frac{|A|}{|P_{u^\bot}A|_{n-1}}+\frac{|B|}{|P_{u^\bot} B|_{n-1}}\right),
\end{equation}
which is \eqref{constrong} for $n=2$. 
The fact that (\ref{conj:strong}) is true in dimension $2$ and in the case when one of the bodies is an Euclidean ball inspires the natural conjecture that (\ref{conj:strong}) holds for zonoids. Recall that zonoids are Hausdorff limits of zonotopes and that zonotopes are Minkowski sums of segments. Our main goal, in this paper, is to study a weaker version of this conjecture: is it true that for two zonoids $A,B $ in $\RL^n$ we have
\begin{equation}\label{weak}
\frac{|A+B|}{|P_{u^\bot} (A+B)|_{n-1}}\geq \frac{|A|}{|P_{u^\bot}A|_{n-1}}?
\end{equation}
Note that \eqref{weak} is again not true for general convex bodies and $n\ge3$. We prove it for zonoids and $n=3$. We  also present a number of equivalent and useful restatements of \eqref{weak}.

The second main contribution of this paper has to do with an analogue of the Pl\"unnecke-Ruzsa \cite{Plu70, Ruz89}  inequality for zonoids, or equivalently, the log-submodularity property of volume on the space of zonoids with respect to Minkowski summation. More precisely, we study the following conjecture from \cite{FMZ22}:
given zonoids $A, B_1, B_2$ in $\RL^{n}$, one has
\begin{equation}\label{eq:submod}
|A|\, |A+B_1+B_2 |  \,\leq\,  |A+B_1|\, |A+B_2|.
\end{equation}
This conjecture was inspired by Bobkov and the second-named author \cite{BM12:jfa} who proved that for convex bodies $A, B_1, B_2$ in $\RL^{n}$, one has
\[
|A|\, |A+B_1+B_2 |  \leq 3^n  |A+B_1|\, |A+B_2|.
\]
Recently, it was proved in \cite{FMZ22} that the constant $3^n$ in the preceding inequality may be replaced by $\varphi^{n}$, where $\varphi=(1+\sqrt{5})/2$ is the golden ratio, that the best constant $c_n$ is lower bounded by $(4/3+o(1))^n$ and that $c_2=1$ and $c_3= 4/3$. 
These observations for dimensions $n\ge3$ imply that, for general convex bodies, it is impossible to have \eqref{eq:submod}.
In this paper, we prove that \eqref{eq:submod} holds in $\RL^3$ for zonoids.

This paper is organized as follows. In Section~\ref{sec:prelim}, we collect background material on mixed volumes.  Numerous equivalent descriptions of log-submodularity, on a given class of convex bodies, are explored in Section~\ref{sec:sublog}.
In Section~\ref{sec:special}, we first present an improved version of $\log$-submodularity in $\RL^2$ by proving a strong version of two dimensional conjecture of T.~Courtade, and we then prove inequality (\ref{weak}) in the special case of paralleletopes. Section~\ref{sec:zonoid}
explores log-submodularity of volume on the class of zonoids-- we first prove that this holds in
$\RL^3$, and then discuss an array of inequalities that are equivalent to log-submodularity
for zonoids in arbitrary dimension (all of which therefore now hold in $\RL^3$).
Along the way, we answer a question of Adam Marcus about Steiner polynomials of zonoids. Finally, in Section~\ref{sec:Lp-z}, inspired by a work of  Brazitikos and McIntyre \cite{BM21:1} on vector-valued Maclaurin inequalities,  we discuss possible extensions of our results to the more general class of $L^p$-zonoids, which appear in the more general $L_p$-Brunn-Minkowski theory.


\vspace{.1in}
\noindent
{\bf Acknowledgments.}
We are indebted to  Shiri Artstein-Avidan, Guillaume Aubrun,  Silouanos Brazitikos, 
Dan Florentin, Dylan Langharst, Ivan Soprunov, and Ramon Van Handel, for a number of  valuable discussions and suggestions.

\section{Preliminaries on mixed volumes}
\label{sec:prelim}


In this section, we introduce basic notations and collect essential facts and definitions 
from convex geometry that are used in the paper. As a general reference on the theory, we use \cite{Sch14:book}.
We write $\langle x,  y\rangle$  for the inner product of vectors $x$ and $y$ in $\RL^n$  and  by $|x|$ the Euclidean norm of a vector $x \in \RL^n$. The closed Euclidean ball in $\RL^n$ is denoted by $B_2^n$, and its boundary by $S^{n-1}$. We  also denote by $e_1, \dots, e_n$ the standard orthonormal basis in $\RL^n$. Moreover, for any set in $K \subset \RL^n$, we denote its boundary by $\partial K$.  A compact set $K$ in $\RL^n$ is called star-shaped if, for every $x\in K$, the segment $[0,x]$ is a subset of $K$; its radial function $\rho_K$ is then defined by $\rho_K(x)=\sup\{a; ax \in K\}$. When $0$ belongs to the interior of $K$ then $\|x\|_K=\rho_K^{-1}(x)$
is the Minkowski functional of $K$. A convex body is a convex, compact set with non-empty interior. For a convex body $K$, we define its support function  by $h_K(x) = \max_{y\in K} \langle x, y\rangle$.

We write $|K|_m$ for the $m$-dimensional Lebesgue measure (volume) of a measurable set $K \subset \RL^n$, where $m = 1, . . . , n$ is the dimension
of the minimal affine space containing $K$, we  often use the  shorten notation  $|K|$ for $n$-dimensional volume.  
From  \cite[Theorem 5.1.6]{Sch14:book}, for any compact convex sets $K_1,\dots, K_r$ in $\RL^n$ and any non-negative numbers $t_1, \dots, t_r$, 
one has
\be\label{eq:mvf}
\left|t_1K_1+\cdots+t_rK_r\right|=
\sum_{i_1,\dots,i_n=1}^rt_{i_1}\cdots t_{i_n}V(K_{i_1},\dots, K_{i_n})
\ee
for some non-negative numbers $V(K_{i_1},\dots, K_{i_n})$, which are  the mixed volumes of $K_{i_1},\dots, K_{i_n}$.
We  also often use a two bodies version of (\ref{eq:mvf}):
\be\label{eq:ste}
\left|A+tB\right|=
\sum_{k=0}^n {n \choose k} t^k V(A[n-k],B[k]),
\ee
for any $t>0$ and compact, convex sets $A, B$ in $\RL^n$, where for simplicity we use notation $A[m]$ for a convex set $A$ repeated $m$ times. 
Mixed volumes satisfy a number of extremely useful inequalities. The first one is the Brunn-Minkowski inequality 
$
|K+L|^{1/n}\ge |K|^{1/n}+|L|^{1/n},$ 
whenever  $K, L$ and $K+L$ are measurable. A direct consequence are   Minkowski's first inequality
\be\label{M1}
V(L, K[n-1])\ge |L|^{1/n}|K|^{(n-1)/n},
\ee
and  Minkowski's second inequality
\be\label{M2}
V(L, K[n-1])^2\ge |K|V(L[2],K[n-2]),
\ee
for two convex, compact subsets $K$ and $L$ in $\RL^n$. 
We will  use the classical integral representation for the mixed  volume:
\begin{equation}\label{eq:mixvol}
V(L, K[n-1])=\frac{1}{n}\int_{S^{n-1}} h_L(u) dS_K(u),
\end{equation}
where  $S_K$ is the surface area measure of $K$ \cite{Sch14:book}.
Mixed volumes are also  useful to study the volume of the orthogonal projections of convex bodies. Let $E$ be an $m$-dimensional subspace of $\RL^n$, for $1\le m \le n$ and let $P_E:\RL^n \to E$ be the orthogonal projection onto $E$.  Then for any convex body $K$ we have 
\be\label{proj}\label{eq:proj}
|U|_{n-m} |P_EK|_m={{n}\choose{m}}V(K[m], U[n-m]),
\ee
where $U$ is any convex body in the subspace $E^\perp$ orthogonal to $E$. It follows from (\ref{eq:proj}) that for any   orthonormal system $u_1, \dots, u_r,$ $1\le r \le n$ we get 
\begin{equation}\label{eq:projvec}
\left|P_{[u_1,\dots,u_r]^\bot}K\right|_{n-r}=\frac{n!}{(n-r)!}V(K[n-r],[0,u_1],\dots, [0,u_r]).
\end{equation}
For example, denote by $u^\perp = \{x \in \RL^n : \langle x, u \rangle =0\}$ the hyperplane orthogonal to a  vector $u\in S^{n-1}$, we obtain
\be\label{eq:proj1}
|P_{u^\perp} K|_{n-1}= n V(K[n-1], [0, u]).
\ee
Another useful formula is connected with the computation of surface area and mixed volumes:
\be\label{surface}
|\partial K|= nV(K[n-1], B_2^n),
\ee
where by $|\partial K|$ we denote the surface area of the convex body  $K$ in $\RL^n$.

A polytope which is the Minkowski sum of
finitely many line segments is called a zonotope. 
Limits of zonotopes in the Hausdorff metric are called zonoids, see \cite[Section 3.2]{Sch14:book} for details. 
Consider zonotopes  $Z_j=\sum_{i_j=1}^{N_j} [0, w_{i_jj}]$, where, $j=1,\dots, n$ and $w_{i_jj} \in \RL^n$. Using the linearity of mixed volumes,  we get that  
\begin{equation}\label{eq:formulazonoids}
V(Z_1, \dots, Z_n)=\sum V([0, w_{i_11}],\dots [0, w_{i_nn}])= \frac{1}{n!}\sum|\det(\{w_{i_jj}\}_{j=1}^n)|,
\end{equation}
where the sums runs over the integers $i_j \in \{1, \dots, N_j\}$, for $j \in \{1, \dots, n\}$. Notice that, if $Z_1=[0,u]$, with $u\in S^{n-1}$, we may use the basic properties of determinants to show that
$$
V([0, u], Z_2, \dots, Z_n)=\frac{1}{n!}\sum|\det(\{P_{u^\perp}w_{i_j j}\}_{j=2}^n)|,
$$
where the sum is  taken over all integer $i_j \in \{1, \dots, N_j\}$, for $j \in \{2, \dots, n\}$. We  mainly use the following particular case: for $Z=[0,u_1]+\cdots+[0,u_m]$, then 
\begin{align}\label{form-vol-zo}
|Z|=\!\!\!\sum_{1\le i_1<\cdots<i_n \le m}&|\det(u_{i_1},\dots,u_{i_n})|\\
&|P_{e^\bot}Z|=\!\!\!\sum_{1\le i_2<\cdots<i_n\le m}|\det(P_{e^\bot}u_{i_2},\dots,P_{e^\bot}u_{i_n})|.\nonumber
\end{align}

\section{Equivalent forms of the log-submodularity of volume}
\label{sec:sublog}

\subsection{Submodularity}
\label{ss:smod}

Let us recall the notion of submodular set function and remind some basic properties, see \cite{MG19, Top98:book, FH05, FMZ22} for more information on this subject and the proofs of the theorems. We denote ${[n]}=\{1,\cdots,n\}$ and let $2^{[n]}$  be the family of subsets of ${[n]}$.

\begin{defn}\label{def:supermod}
A set function $F:2^{[n]}\ra\RL$ is {\it submodular} if 
\be\label{supmod:defn}
F(\setS\cup\setT)+F(\setS\cap\setT) \leq F(\setS) + F(\setT)\quad \mbox{for all subsets $\setS, \setT$ of $[n]$}.
\ee
\end{defn}

Submodularity is closely related to a partial ordering on hypergraphs  \cite{Eme75, BL91, BB12}.
Let $\calM(n,m)$ be the following family of 
multi-hypergraphs: each consists of non-empty subsets $\setS_i$ of $[n]$,
$\sum_i |\setS_i|= m$, with $\setS_i=\setS_j$ allowed.
Consider a given multi-hypergraph 
$\collS = \{\setS_1, \dots, \setS_l\} \in \calM(n,m)$.
Take any pair of non-nested sets $\{\setS_i,\setS_j\}\subset \collS$
and let $\collS' = \collS(i,j)$ be obtained from $\collS$ by replacing $\setS_i$ and $\setS_j$ 
by $\setS_i \cap \setS_j$ and $\setS_i \cup \setS_j$, 
keeping only $\setS_i \cup \setS_j$ if $\setS_i \cap \setS_j = \emptyset$. 
$\collS'$ is called an  elementary compression of $\collS$. The result of a sequence of
elementary compressions is called a compression.
Define a partial order on $\calM(n,m)$ by setting $\calA > \calA'$ if $\calA'$ is a
compression of $\calA$. 
Using transitivity of the partial order and reducing to elementary compression, we get the following theorem (see \cite{BB12}).
\begin{thm}\label{thm:compr}
Suppose $F$ is a submodular function on $[n]$.
Let $\calA$ and $\calB$ be finite multi-hypergraph of subsets of $[n]$, with $\calA > \calB$.
Then
\ben
\sum_{\setS\in\calA} F(\setS) \geq \sum_{\setT\in\calB} F(\setT) .
\een
\end{thm}

For every multi-hypergraph $\calA \in \calM(n,m)$ there is a unique minimal multi-hypergraph
$\calA^{\#}$ dominated by $\calA$ consisting of the sets
$\setS^{\#}_j = \{i \in [n] : i \text{ lies in at least } j \text{ of the sets } \setS \in \calA\}$. One implication of Theorem~\ref{thm:compr} is that submodular functions $F$ with $F(\phi)=0$ are ``fractionally subadditive'' (see, e.g., \cite{MT10})-- this property has also been investigated in connection with volumes of Minkowski sums \cite{BMW11, FMMZ16, BM21}.

We also have a notion of submodularity on the positive octant of the Euclidean space.

\begin{defn}\label{def:supermod-Rmaxmin}
A function $f:\RL_{+}^{n} \ra\RL$ is submodular if, for any $x,y \in \mathbb{R}_+^n$,
\ben
f(x\vee y) + f(x \wedge y) \le f(x) + f(y), 
\een
where $x \vee y$ (resp. $x \wedge y$)  denotes the componentwise maximum (resp. minimum) of $x$ and $y$.
\end{defn}

The next very simple lemma connects submodularity for functions defined on $\RL_{+}^n$ to submodularity for set functions. For a set $S\subset \{1,\dots, n\}$, let $\one_S$ be the vector in $\RL^n$ such that for each $i\in \{1,\dots, n\}$, the $i$-th coordinate of $\one_S$ is 1, for $i\in S$ and $0$, for $i\notin S$.

\begin{lem}\label{lem:set-rn}
If $f:\RL_{+}^n\ra\RL$ is submodular, and we set $F(S):=f(\one_S)$ for each $S\subset \{1,\dots, n\}$,
then $F$ is a submodular set function.
\end{lem}


The fact that submodular functions are closely related to functions
with decreasing differences is classical (see, e.g., \cite{MG19} or  \cite{Top98:book}, which describes
more general results involving arbitrary lattices).  We  denote by  $\partial^2_{i,j}$ the second derivative with respect to coordinates $i, j$.

 \begin{prop}\label{prop:diff}
 Let $f:\RL_+^n\ra\RL$ be a $C^2$ function. Then $f$ is submodular if and only if 
$ \partial_{i,j}^2 f\leq 0$ on $\RL_+^n$, for every $i\neq j$.
 \end{prop}

\subsection{Classes closed under sums and dilation}
\label{ss:mlc-equiv}

Let us now show the various forms that the notions of submodularity and log-submodularity can have applied to volume of convex compact sets.


\begin{thm}\label{thm:logsubmod-equiv}
Consider a collection $\cal{K}$ of compact convex sets in $\RL^n$ stable by sums and dilation. Then the following statements are equivalent:
\begin{enumerate}
\item For every $m\ge 1$ and every $A, B_1,\dots, B_m$ in ${\cal K}$
\be\label{set}
|A|^{m-1}\left|A+\sum_{i=1}^mB_i\right|\leq \prod_{i=1}^{m} |A+B_i|.
\ee
\item For every $m\ge 1$ and  $A, B_1,\dots, B_m\in\cal{K}$, the function $\bar{w}:2^{[m]}\ra [0,\infty)$ defined by
\be\label{level2}
\bar{w}(\setS)=\log \bigg|A+\sum_{i\in\setS} B_i \bigg|
\ee
for each $\setS\in2^{[m]}$, is a submodular.
\item For every $m\ge 1$ and  any multi-hypergraphs $\calA, \calB$ on $[m]$ with $\calA>\calB$,
\ben
\sum_{\setS\in\calA} \bar{w}(\setS) \geq \sum_{\setT\in\calB} \bar{w}(\setT),
\een
where $\bar{w}$ is defined by (\ref{level2}).
\item For every $m\ge 1$ and  every $A, B_1,\dots, B_m\in\cal{K}$, the function $w:\Rpl^{m}\ra [0,\infty)$ defined, for  $x=(x_1,\dots,x_m)\in\Rpl^m$, by
\be
w(x)=\log \bigg|A+\sum_{i=1}^m x_i B_i \bigg|, 
\ee
 is submodular.

\item For every $A, B_1,B_2\in\cal{K}$
\ben
|A| V(A[n-2], B_1, B_2) \leq \frac{n}{n-1} V(A[n-1], B_1) V(A[n-1], B_2) .
\een

\end{enumerate}
\end{thm}

\begin{proof}
2. $\implies$ 1. This follows directly from the definition. \\
1. $\implies$ 2.: We use (\ref{set}) with $A+\sum_{i\in S\cap T}B_i$, $\sum_{i\in S\setminus T}B_i$ and $\sum_{i\in T\setminus S}B_i$.\\
2. $\Longleftrightarrow$ 3.: This follows from Theorem~\ref{thm:compr} applied to $-\bar{w}$.\\
2. $\Longleftrightarrow$ 4.: This comes from Lemma~\ref{lem:set-rn} and the fact that $\cal{C}$ is dilation invariant.\\
4. $\Longleftrightarrow$ 5.: We apply
Proposition~\ref{prop:diff} to $f=-w$. Let $K_x=A+\sum_i x_i B_i$ and $v(x)=|K_x|$. Then 
\ben
\partial_j v(x) = \lim_{\eps\ra 0} \frac{|K_x+\eps B_j|- |K_x|}{\eps}= nV(K_x[n-1], B_j) ,
\een
and for $k\neq j$
\ben\begin{split}
\partial_{j,k}^2 v(x)
&= n\partial_k V(K_x[n-1], B_j) \\
&= n \lim_{\eps\ra 0} \frac{V( (K_x+\eps B_k)[n-1], B_j)- V( K_x[n-1], B_j)}{\eps} \\
&= n(n-1) V(K_x[n-2], B_k, B_j) .
\end{split}\een
By Proposition~\ref{prop:diff}, the submodularity of $w$ is equivalent to $\partial_{j,k}^2w(x)\le0$. But $w=\log(v)$ so
\ben
\partial_{j,k}^2w(x)=\frac{v\dou_{jk}v -\dou_{j}v\dou_{k}v }{v^2},
\een
which is negative if and only if  
\ben
|K_x| V(K_x[n-2], B_j, B_k) \leq \frac{n}{n-1} V(K_x[n-1], B_j) V(K_x[n-1], B_k).
\een
Thus plugging $x=0$ we get $4. \implies 5.$ and using 4. with $K$ instead of $A$ gives $5. \implies 4.$
\end{proof}


We note that it is enough to check property 1. in Theorem \ref{thm:logsubmod-equiv}, just in the case $m=2$. Indeed the case of general $m>2$ follows by an iteration argument.

We also note that, in dimension $2$, property 5. of Theorem \ref{thm:logsubmod-equiv} holds for any convex bodies  by the classical local version of Alexandrov's inequality that was proved by W. Fenchel (see \cite{Fen36}, also \cite{Sch93:book} and discussion in Section \ref{ss:dim2} below)  and further generalized in \cite{FGM03, AFO14, SZ16}: 
for any convex compact sets $A,B_1,B_2$ in $\RL^n$ we have 
\begin{equation}\label{eq:constanttwo}
|A| V(A[n-2], B_1, B_2) \leq 2 V(A[n-1], B_1) V(A[n-1], B_2).
\end{equation}
The constant $2$ is sharp in any dimension (see \cite{GHP02} and \cite{FMZ22}). This shows also that log-submodularity  doesn't hold in the set of compact convex sets in $\RL^n$, for $n\ge3$.
From \eqref{eq:constanttwo} and Theorem \ref{thm:logsubmod-equiv}, the following theorem holds. 

\begin{thm}\label{th:dim2-sums} 
For every $A, B_1, B_2$ convex compact in $\RL^2$ it holds
\[
|A|\,|A+B_1+B_2|\leq |A+B_1|\,|A+B_2|.
\]
\end{thm}

It is easy to see that 1. from Theorem \ref{thm:logsubmod-equiv} works well with direct sums, more precisely if $A^1,B^1_1, \dots, B_m^1 \subset \RL^{n_1}$  and  $A^2,B^2_1, \dots, B_m^2 \subset \RL^{n_2}$ satisfy  (\ref{set}), then $A=A^1\times A^2\subset\RL^{n_1+n_2}$ and $B_i=B^1_i\times B_i^2$ in $\RL^{n_1+n_2}$, $i=1, \dots,m$ satisfies (\ref{set}). This fact 
can be used to create different classes of compact convex sets stable by sum and dilations which satisfy the properties of Theorem \ref{thm:logsubmod-equiv}.

\begin{rem}
Minkowski's second inequality gives that, for any compact convex sets $A$ and $B$, one has
\be\label{eq:af-cons}
|A| V(A[n-2], B[2])\leq V(A[n-1], B)^2. 
\ee
Theorem \ref{thm:logsubmod-equiv} implies that if 
$\calK$ is a class of convex bodies on which log-submodularity holds (property 1. of Theorem \ref{thm:logsubmod-equiv}), then for bodies in $\calK$, a Fenchel type inequality similar to \eqref{eq:constanttwo} holds with a dimensional factor $\frac{n}{n-1}$ instead of $2$.

Let us, also, note that the same proof shows that for a fixed compact convex set $A$ not necessarily belonging to $\cal{K}$, if, for every $B_1, B_2$ in ${\cal K}$, one has
\be
|A|\,|A+B_1+B_2|\leq |A+B_1|\,|A+B_2|,
\ee
then, for every $B_1,B_2\in\cal{K}$, one has
\ben
|A| V(A[n-2], B_1, B_2) \leq \frac{n}{n-1} V(A[n-1], B_1) V(A[n-1], B_2) .
\een
\end{rem}

\subsection{Classes closed under linear transformations}
\label{ss:lin-equiv}

\begin{thm}\label{th:linear}
Let $\L$ be a class of a compact convex sets in $\RL^n$ stable under any linear transformations.  The following are equivalent.
\begin{enumerate}
\item $|A|\,|\partial(A+[0,u])|\le |\partial A|\,|A+[0,u]|$  for any $A\in\L$ and any $u\in\RL^n$.
\item $|A|\,|\partial( P_{u^\bot}A)|\le|\partial A|\, |P_{u^\bot}A|_{n-1}$, for any $A\in\L$  and any $u\in S^{n-1}$.
\item $|A|\,|P_{[u,v]^\bot}A|_{n-2}\sqrt{1-\langle u, v\rangle^2}\le   |P_{u^\bot}A|_{n-1}|P_{v^\bot}A|_{n-1}$, for any $A\in\L$  and any $u,v\in S^{n-1}$.
\item $|A+[0,u]+[0,v]|\,|A|\le |A+[0,u]|\,|A+[0,v]|$ for any $A\in\L$ and any $u,v\in\RL^n$.
\item $|A|V(A[n-2],Z_1,Z_2)\le  \frac{n}{n-1} V(A[n-1],Z_1)V(A[n-1],Z_2)$, for any $A\in\L$ and any $Z_1, Z_2$ zonoids.
\item For any $A\in\L$ and any $u,v\in\RL^n$, let us define $P(t)=|A+t([0,u]+[0,v])|$, for $t>0$. Then $P(t)$ is the restriction to $\RL_+$ of a polynomial on $\RL$, which has only real roots.
\end{enumerate}

\end{thm}

\begin{proof}
$1.\iff 2.$: This was observed in \cite{AFO14}. It is true even for fixed $A$ and $u$ and follows from the identities $|A+[0,u]|=|A|+|P_{u^\bot}A|$ and 
$$|\partial(A+[0,u])|=nV((A+[0,u])[n-1], B_2^n)=
|\partial A|+|\partial (P_{u^\bot}A)|.
$$
$2. \implies 3.$: Define $T_\eps:\RL^n\to\RL^n$ by $T_\eps x=\eps x+\langle x,v\rangle v$. Notice that $T_0(B_2^n)=[-v,v]$. From $(2)$ applied to $T_\eps^{-1}A$ and $T_\eps^{-1}u$ we have 
$$
|T_\eps^{-1}A|V(T_\eps^{-1}A[n-2],[0,T_\eps^{-1}u], B_2^n)\le\frac{n}{n-1}V(T_\eps^{-1}A[n-1], B_2^n)V(T_\eps^{-1}A[n-1], [0,T_\eps^{-1}u]).
$$
Thus
\[
|A|V(A[n-2],[0,u], T_\eps B_2^n)\le\frac{n}{n-1}V(A[n-1], T_\eps B_2^n)V(A[n-1], [0,u]).
\]
When $\eps\to 0$, we get 
\begin{eqnarray}\label{eq:mixedvol}
|A|V(A[n-2],[0,u], [0,v])\le\frac{n}{n-1}V(A[n-1],[0,v])V(A[n-1], [0,u]).
\end{eqnarray}
From (\ref{eq:proj}) we get   $|P_{u^\bot}A|_{n-1}=nV(A[n-1], [0,u])$ and
\[
|P_{[u,v]^\bot}A|_{n-2}\sqrt{1-\langle u, v\rangle^2}=n(n-1)V(A[n-2],[0,u], [0,v]).
\]
$3. \iff 4.$: We may assume that $u$ is not colinear with $v$. Applying a linear transformation to $A, u$ and $v$, we may assume that $u, v$ are orthonormal. Expanding both sides of the inequality in 4. and using \eqref{eq:mvf}, we get 3.\\
$3. \implies 5.$: As noticed above, $3.$ is equivalent to (\ref{eq:mixedvol}). From the linearity of mixed volumes, we deduce that for every zonotopes $Z_1$ and $Z_2$, one has 
\begin{eqnarray}\label{eq:mixedvolzo}
|A|V(A[n-2],Z_1,Z_2)\le  \frac{n}{n-1} V(A[n-1],Z_1)V(A[n-1],Z_2).
\end{eqnarray}
Taking limits, we conclude that \eqref{eq:mixedvolzo} is valid for every zonoids $Z_1, Z_2$.\\
$5. \implies2.$: Applying (\ref{eq:mixedvolzo}) to $Z_1=[0,u]$ and $Z_2=B_2^n$ and using that $V(A[n-2],[0,u],B_2^n)=\frac{1}{n(n-1)}|\partial (P_{u^\bot}A)|$,  $V(A[n-1],[0,u])=\frac{1}{n}|P_{u^\bot}A|_{n-1}$ and $V(A[n-1],B_2^n)=\frac{1}{n}|\partial A|$, we conclude.\\
$3.\iff 6.$: We may assume that $u$ is not colinear with $v$. Applying a linear transformation, to (3) and (6), it is enough to assume that  $u, v$ are orthonormal. Then
$$
P(t)=|A|+ t(|P_{u^\bot}A|_{n-1} +|P_{v^\bot}A|_{n-1})+
|P_{[u,v]^\bot}A|_{n-2}t^2.
$$
The equation 
$$
|A|+ t(|P_{u^\bot}A|_{n-1} +|P_{v^\bot}A|_{n-1})+
|P_{[u,v]^\bot}A|_{n-2}t^2 =0
$$
has real roots is equivalent to 
\be\label{arith}
\left(\frac{|P_{u^\bot}A|_{n-1} +|P_{v^\bot}A|_{n-1}}{2}\right)^2 \ge |A| |P_{[u,v]^\bot}A|_{n-2},
\ee
which follows immediately from 3.. To show that $6. \implies 3.$ assume (\ref{arith}) is true for all $A$ in $\L$ and  $u,v \in S^{n-1}$. Consider the linear  operator $T$ such that $Tu =t u$ and $Tv=t^{-1} v$ for $t>0$ and $Tx=x$ for $x \in [u,v]^\perp$.
Taking $t=(|P_{u^\bot}A|_{n-1}/|P_{v^\bot}A|_{n-1}|)^{1/2}$ we get  
$$
|P_{v^\bot} T A|_{n-1} = |P_{u^\bot} T A|_{n-1}.
$$
Applying (\ref{arith}) to $TA$,   we get
$$
|P_{u^\bot}TA|_{n-1} |P_{v^\bot}TA|_{n-1} \ge |TA| |P_{[u,v]^\bot}TA|_{n-2}.
$$
Since $|TA|=|A|$, $|P_{[u,v]^\bot}TA|_{n-2}=|P_{[u,v]^\bot}A|_{n-2}$  
and 
\[|P_{u^\bot}TA|_{n-1} |P_{v^\bot}TA|_{n-1} =|P_{u^\bot}A|_{n-1} |P_{v^\bot}A|_{n-1}
\]
we get 3.
\end{proof}

\begin{rem}
It was proved in \cite{SZ16} that 5. in Theorem \ref{th:linear} is satisfied when $A$ is a simplex (actually, even without constant $n/(n-1)$), thus all of the properties  in Theorem \ref{th:linear} are true for simplices. Actually, the inequalities of statements 2. and 3. hold with an extra factor $\frac{n-1}{n}$ on the right hand side. 
\end{rem}
\begin{rem}
Notice that the inequality \eqref{eq:constanttwo} shows that the property $5.$ of Theorem \ref{th:linear} holds true for the class $\mathcal{L}$ of compact convex sets in $\RL^2$ and doesn't hold in the class of compact convex sets in $\RL^n$, for $n\ge3$. 
\end{rem}

 \section{Some Special Cases}
 \label{sec:special}
 
\subsection{An improved inequality in $\RL^2$}
\label{ss:dim2}

Inspired by an analogous result \cite[Theorem 3]{Cou18} in Information Theory, T.~Courtade asked if 
\begin{equation}\label{cortn}
|B|^{1/n}|C|^{1/n}+|A|^{1/n}|A+B+C|^{1/n}\leq |A+B|^{1/n}|A+C|^{1/n}
\end{equation}
for $A=B_2^n$ being the Euclidean ball, and any convex bodies $B,C$  in $\RL^n$. 
Here we confirm Courtade's conjecture in $\RL^2$ in a more general setting.

\begin{thm}
Consider convex bodies $A,B, C\subset \RL^2$, then
\begin{equation}\label{cort2}
|A|^{1/2}|A+B+C|^{1/2}+|B|^{1/2}|C|^{1/2}\leq |A+B|^{1/2}|A+C|^{1/2}.
\end{equation}
\end{thm}
\begin{proof}
The main tool to prove the above inequality is the following classical inequality of Fenchel that we have already used.  We will need now to use the most general form of this inequality (see \cite[(7.69) pp. 401]{Sch14:book}):
$$
(|A|V(B,C)-V(A, B)V(A,C))^2 \le (V(A,B)^2-|A||B|)
(V(A,C)^2-|A||C|).
$$
Note that the above can be rewritten as
\begin{equation}\label{bon}
|C|V(A,B)^2+|B|V(A,C)^2+|A|V^2(B,C)-|A||B||C|- 2V(A, B)V(A,C)V(B,C)\le 0
\end{equation}
Squaring both sides of (\ref{cort2})  we get
$$
2\big(|A||A+B+C||B||C|\big)^{1/2}+ |A||A+B+C|+|B||C|\leq |A+B||A+C|
$$
We use (\ref{eq:mvf}) to rewrite above as
$$
\left(|A||B||C||A+B+C|\right)^{1/2}+|A|V(B,C)
\le
2V(A,B)V(A,C)+V(A,B)|C|+|B|V(A,C).
$$
Using (\ref{eq:constanttwo})  we may rewrite the above inequality  as
$$
|A||B||C||A+B+C|
\le
\left(2V(A,B)V(A,C)+V(A,B)|C|+|B|V(A,C)- |A|V(B,C)\right)^2.
$$
Consider $rB$, $r\ge0$, instead of $B$:
$$
|A||B||C||A+rB+C|
\le
\left(2V(A,B)V(A,C)+V(A,B)|C|+r|B|V(A,C)- |A|V(B,C)\right)^2.
$$
The above represents a quadratic inequality 
$
\alpha r^2 +\beta r +\gamma \ge 0,
$
with 
$$
\alpha=|B|^2V^2(A,C)-|A||B|^2|C|,
$$
$$
\beta=2\big(2V(A,B)V(A,C)+V(A,B)|C|- |A|V(B,C)\big) |B|V(A,C) - 2|A||B||C|V(B,A+C),
$$
$$
\gamma=\big(2V(A,B)V(A,C)+V(A,B)|C|- |A|V(B,C)\big)^2-|A||B||C||A+C|.
$$
It follows from (\ref{M1}) that $\alpha \ge 0$. It turns out $\beta$ may be negative and thus we need to  show that
$
D=\beta^2-4\alpha \gamma \le 0,
$
which after division by $|B|^2$ becomes
$$
\left[\Big(2V(A,B)V(A,C)+V(A,B)|C|- |A|V(B,C)\Big) V(A,C) - |A||C|V(B,A+C)\right]^2
$$
$$
-\big(V^2(A,C)-|A||C|\big)\Big[\Big(2V(A,B)V(A,C)+V(A,B)|C|- |A|V(B,C)\Big)^2-|A||B||C||A+C| \Big] \le 0.
$$
Simplifying the above inequality and dividing it  by $|A||C|$ we may rearrange the terms to  get that our goal is  to show that
$$
\Big(|C|V(A,B)^2-2V(A,B)V(A,C)V(B,C)+|A|V^2(B,C)\Big)\Big(|A|+2V(A,C)+|C|\Big)
$$
$$
+ \big(V^2(A,C)-|A||C|\big)|B|(|A+C|) \le 0.
$$
Factoring out $|A+C|$ we get that our goal is to show that 
$$
\Big(|C|V(A,B)^2-2V(A,B)V(A,C)V(B,C)+|A|V^2(B,C)+|B|V^2(A,C)-|A||C||B|\Big)
|A+C| \le 0. 
$$
Finally, the above inequality  follows from (\ref{bon}).

\end{proof}

From the failure of log-submodularity on the space of convex bodies for $n\geq 3$
(observed independently by Nayar and Tkocz \cite{NT17} and a subset of the authors \cite{FMZ22}), we know that inequality \eqref{cortn} cannot possibly hold
for $n\geq 3$ if $A$ is an arbitrary convex body. Of course, Courtade's conjecture
could still be true since it only considers the case $A=B_2^n$. We note 
that a weaker version of the conjecture, namely,
$$
|B_2^n||B_2^n+B+C|\leq |B_2^n+B||B_2^n+C|
$$
was proved in \cite[Theorem 4.12]{FMZ22} in the special case 
when $B$ is a zonoid and $C$ is an arbitrary convex body.

\subsection{Parallelotopes in general dimension}
\label{ss:paral}

It is clear that 3. from Theorem~\ref{th:linear} holds for ellipsoids. Indeed, after applying a linear transformation, it reduces to the following inequality: $|B_2^n||B_2^{n-2}|\le |B_2^{n-1}|^2$, which follows from the log-convexity of the Gamma function.
In the next theorem, we  prove that 3. from  Theorem~\ref{th:linear} holds for the class of parallelotopes.

\begin{thm}\label{th:par} Let $A$ be a parallelotope in $\RL^n$ and $u,v \in S^{n-1}$, then
$$
|A|\,|P_{[u,v]^\bot}A|_{n-2}\sqrt{1-\langle u, v\rangle^2}\le   |P_{u^\bot}A|_{n-1}\,|P_{v^\bot}A|_{n-1},
$$
with equality if and only, when $A=\sum_{i=1}^n[a_i, a_i+w_i],$ for some $a_i, w_i$ in $\RL^n$ we have that
$u=\sum_{i\in I}\lambda_i w_i$ and $v=\sum_{i\in I^c}\lambda_i w_i$ for some $I \in \{1,\dots, n\}$ and $\lambda_1,\dots, \lambda_n$ in $\RL^n$.
\end{thm}
\begin{proof}
We  use the representation of the volume of projections using mixed volumes (\ref{eq:proj}) to restate the above statement as
\[
|A|V(A[n-2],[0,u],[0,v])\le \frac{n}{n-1}V(A[n-1],[0,u])V(A[n-1],[0,v]).
\]
Applying a linear transformation, we may assume $A=[0,1]^n=\sum_{i=1}^n [0, e_i]$. Thus
$$
V\left(A[n-1], [0, u]\right) = \frac{1}{n} \sum_{|I|=n -1}|\det(u,(e_i)_{i\in I})|=\frac{1}{n}\sum_{i=1}^n |u_i|,
$$
 $$
V\left(A[n-2], [0, u], [0, v]\right) = \frac{1}{n(n-1)} \sum_{|I|=n -2}|\det(u,v, (e_i)_{i\in I})|= \frac{1}{n(n-1)} \sum_{i<j} |u_i v_j - u_j v_i|.
$$
Finally we need to show
\begin{equation}\label{easy}
 \sum_{i< j} |u_i v_j - u_j v_i|  \le  \sum |u_i|  \sum |v_j|,
\end{equation}
which follows from the triangle inequality. The equality in (\ref{easy}) is only possible if and only if $u_iv_i =0$ for every $i \in \{1,\dots, n\},$ which implies the desired equality case.
\end{proof}


We note that every zonotope $A$ can be seen as an orthogonal projection of a high dimensional cube. Unfortunately, Theorem \ref{th:par} can not be generalized directly to the case of projection of higher co-dimensions (as it is done in 3. from Theorem \ref{thm:zon} below), indeed such direct generalization requires Theorem \ref{th:par} for all zonotopes in place of the cube. Thus we prove this property directly for the case of parallelotopes in the next theorem.

\begin{thm} Let $A$ be a parallelotope in $\RL^n$, then
\begin{equation}\label{eq:inequality}
|A||P_{E\cap F} A| \le |P_EA||P_FA|.
\end{equation}
for any subspaces $E,F$ of $\RL^n$ such that $E^\perp \subset F$.
\end{thm}
\begin{proof}

Notice that (\ref{eq:inequality}) is invariant under application of rotation $S\in O(n)$ to the parallelotope $A$ and subspaces $E$ and $G.$ Thus, without loss of generality, we may assume $A=T\left(\sum_{i=1}^n [0, e_i]\right),$ for some  $T \in GL(n)$ and $E=\{e_1, \dots, e_m\}^\perp$ and $F=\{e_{m+1}, \dots e_{m+j}\}^\perp$. Then
$$
 |P_{E} \left(T \sum_{i=1}^n [0, e_i]\right)|_{n-m}=|  \sum_{i=1}^n [0, P_{E}(Te_i)]|_{n-m}
 = | \sum_{i=1}^n [0, P_E Te_i]+\sum_{k=1}^m[0, e_k]|_{n},
$$
where the last equality follows from  $\sum_{i=1}^n [0, P_{E}(Te_i)]|_{n-m} \subset E=\{e_1, \dots, e_m\}^\perp$. Thus, 
\begin{align} \label{projcubef}
|P_{E} \left(T \sum_{i=1}^n [0, e_i]\right)|_{n-m} =  & \sum_{|I|=n-m} |\det(e_1, \dots, e_m, \{P_E T e_i\}_{i \in I}|_{n}\nonumber \\=&\sum_{|I|=n-m} |\det(e_1, \dots, e_m, \{Te_i\}_{i \in I}|_{n}\nonumber\\
=& |\det(T)| \sum_{|I|=n-m} |\det(T^{-1}e_1, \dots, T^{-1} e_m, \{ e_i\}_{i \in I}|_{n} \nonumber \\=& |\det(T)| \sum_{|J|=m}|\det(\{w_i^J\}_{i=1}^m)|,
\end{align}
where we denote by $w_i=T^{-1} e_i,$ $i=1,\dots, n$ and by $u^J$ we denote the orthogonal projection of vector $u$ onto  the $span\{ e_i\}_{i \in J}$. We  apply (\ref{projcubef}) to get that 
\ben
|P_{E\cap F} \sum_{i=1}^n [0, Te_i]|_{n-j-m} &=& |\det(T)| \sum_{|N|=m+j}|\det (\{w_i^N\}_{i=1}^{m+j})|\\
&=&|\det(T)| \!\!\!  \sum_{|N|=m+j}
 \left|\sum_{|I|=m, I \subset N}\!\!\!\!\!\! \varepsilon (N,I) \det(\{w_i^I\}_{i \le m}) \det(\{w_i^{N\cap I^c}\}_{i > m})\right|\!, 
 \een
 where in the last step we have used the Laplace  formula. Finally,
 \ben
|P_{E\cap F} \sum_{i=1}^n [0, T e_i]|_{n-j-m} &\le&    |\det(T)|  \sum_{N=m+j}
 \sum_{|I|=m,  I\subset N} |\det(\{w_i^I\}_{i \le m})|| \det(\{w_i^{N\cap I^c}\}_{i > m})| \\  
 &= &  |\det(T)| 
 \sum_{|I|=m}  \sum_{N=m+j, I \subset J}|\det(\{w_i^I\}_{i \le m}) ||\det(\{w_i^{N\cap I^c}\}_{i > m})| \\
 &= &  |\det(T)| 
 \sum_{|I|=m}  \sum_{I \cup J, |J|=j, I \cap J=\emptyset}|\det(\{w_i^I\}_{i \le m})|| \det(\{w_i^{J}\}_{i > m})| \\
 &= &  |\det(T)|
 \sum_{|I|=m}  |\det(\{w_i^I\}_{i \le m})|   \sum_{ |J|=j,}| \det(\{w_i^{J}\}_{i > m})| \\
 &= & |\det(T)|^{-1} |P_E\left(T \sum_{i=1}^n [0, e_i]\right)| \, |P_F\left(T \sum_{i=1}^n [0, e_i]\right)|,
 \een 
where the last equality, again, follows from (\ref{projcubef}).
\end{proof}

\section{Inequalities for zonoids}
\label{sec:zonoid}

\subsection{Zonoids in $\RL^3$}
\label{ss:zonoid3}
 
Zonoids form a natural class of bodies which is stable under addition and linear transformations.
  In this section, we  confirm property (4) from Theorem \ref{th:linear} in the class of three dimensional zonoids. Thus, using Theorem \ref{thm:hope}, we get that all properties described in Theorems \ref{thm:logsubmod-equiv} and  \ref{th:linear} are true for this class.
 \begin{thm}\label{thm:hope} Let $A$ be a zonoid in $\RL^3$ and  $u,v\in S^2$. Then
  \begin{equation}\label{eq:hope}
  |A|_3|P_{[u,v]^\bot}A|_{1}\sqrt{1-\langle u, v\rangle^2}\le   |P_{u^\bot}A|_{2}|P_{v^\bot}A|_{2}.
  \end{equation} 
 \end{thm}
 \begin{proof} Using (\ref{eq:projvec}), inequality (\ref{eq:hope}) is equivalent to 
\begin{equation}\label{eq:zonoidstrong1}
  |A|_3|V(A,[0,u],[0,v])\le \frac{3}{2} V(A[2],[0,u])V(A[2],[0,v]).
\end{equation} 
We assume that $u,v$ are linearly independent (otherwise the inequality is trivial) and note that it is enough to prove (\ref{eq:zonoidstrong1})   in the case of $u=e_1$, $v=e_2$ and any zonoid $A$. Indeed, the more general case then follows by applying the inequality to $T^{-1} A,$ where $T\in GL(3)$ is such that $Te_1=u$ and $Te_2=v$.  Thus our goal is to prove that, for any zonoid $A\subset \RL^3$,
\begin{equation}\label{eq:zonoidstrong2}  |A|_3|P_{[e_1,e_2]^\bot}A|_{1} \le   |P_{e_1^\bot}A|_{2}|P_{e_2^\bot}A|_{2}.
  \end{equation} 
By approximation, it is enough to prove
(\ref{eq:zonoidstrong2}) when $A$ is a zonotope. Suppose that $A=\sum_{i=1}^M [0, u_i],$ where $u_i=(x_i, y_i, z_i)\in\RL^3$. Using \eqref{form-vol-zo}, we get that (\ref{eq:zonoidstrong2}) is equivalent to
\begin{align}\label{eq:matr}
\sum_{1 \le i < j < k \le M}  \left |\det\!\!\begin{pmatrix}
x_i & x_j & x_k \\
y_i & y_j & y_k \\
z_i & z_j & z_k
\end{pmatrix}\right|  \sum_{i=1}^M |z_i|  \le \!\!\!\sum_{1 \le i < j  \le M} \left |\det\!\!\begin{pmatrix}
y_i & y_j \\
z_i & z_j 
\end{pmatrix}\right|\sum_{1 \le i < j \le M} \left |\det\!\!\begin{pmatrix}
x_i & x_j  \\
z_i & z_j 
\end{pmatrix}\right|
\end{align}
We consider $y_1,\dots, y_M$ and $z_1,\dots, z_M$ as fixed, we write $x=(x_1,\dots,x_M)\in\RL^M$ and we define  $f,g:\RL^M\to\RL$ by 
\[
f(x)=\sum_{1 \le i < j \le M} \left |\det\begin{pmatrix}
x_i & x_j  \\
z_i & z_j 
\end{pmatrix}\right|\quad\hbox{and}\quad g(x)=\sum_{1 \le i < j < k \le M}  \left |\det\begin{pmatrix}
x_i & x_j & x_k \\
y_i & y_j & y_k \\
z_i & z_j & z_k
\end{pmatrix}\right|.
\]
We note that $f$ and $g$ are piecewise affine and convex  with respect to $x_i$, for any $1\le i\le M$. We use the following elementary observation. Let $\psi:  \RL \to \RL$ be a convex function. Fix some positive numbers $\{a_j\}_{j=1}^K$, $d\in\RL$ and real numbers $\{c_j\}_{j=1}^K$. To prove that, for all $t\in \RL$,
$$
\psi(t) \le \varphi(t):=\sum_{j=1}^K a_i|t +c_i| +d, 
$$
it is enough to prove the inequality at all critical points $t=-c_i$ of $\varphi$ and at the limit $t\to \pm \infty.$

We shall apply the above argument inductively to $f$ and $g$ as  functions of $x_i$, for  $i \in \{1,\dots, M\}$, successively, with $x_j$ fixed for $j\neq i$.

We start with $x_1$ and  first check the limiting behavior  at infinity (we  also prove the limit at infinity argument as a part of a more general statement below).  We note that, as $x_1 \to \infty$, the left hand side of (\ref{eq:matr}) behaves like
$$
 |x_1|
\left(\sum_{1  < j < k \le M}  \left |\det\begin{pmatrix}
 y_j & y_k \\
 z_j & z_k
\end{pmatrix}\right|\right)\left(\sum_{j=1}^M |z_j| \right) 
$$
and the right hand side of (\ref{eq:matr})  behaves like 
$$
|x_1|\left(\sum_{1 \le   j <k \le M} \left |\det\begin{pmatrix}
y_j & y_k \\
z_j & z_k 
\end{pmatrix}\right|\right)\left(\sum_{j=2}^M  |z_j| \right).
$$
Thus, (\ref{eq:matr}) becomes
$$
\left(\sum_{1  < j < k \le M}  \left |\det\begin{pmatrix}
 y_j & y_k \\
 z_j & z_k
\end{pmatrix}\right|\right)\left(\sum_{j=1}^M |z_j| \right)
\le \left(\sum_{1 \le j < k  \le M} \left |\det\begin{pmatrix}
y_j & y_k \\
z_j & z_k 
\end{pmatrix}\right|\right)\left(\sum_{j=2}^M  |z_j| \right)
$$
or
$$
\left(\sum_{1  < j < k \le M}  \left |\det\begin{pmatrix}
 y_j & y_k \\
 z_j & z_k
\end{pmatrix}\right|\right) |z_1| 
\le \left(\sum_{1 < k  \le M} \left |\det\begin{pmatrix}
y_1 & y_k \\
z_1 & z_k 
\end{pmatrix}\right|\right)\left(\sum_{j=2}^M  |z_j| \right).
$$
The above equation is exactly the $\RL^2$ analog of (\ref{eq:hope}), with $A=\sum_{i=2}^M[0, (y_i, z_i)]$; $u=[1,0]$ and $v=(y_1, z_1)/|(y_1, z_1)|$, thus it holds.

Our next goal is to study the critical points of $f$ with respect to $x_1$, which satisfy 
$$
\det\begin{pmatrix}
x_1 & x_j  \\
z_1 & z_j 
\end{pmatrix} =0, \mbox{ for }   z_j \not=0, \,\,\, 2 \le j \le M.
$$
 When $x_1$ is a solution of the above equation, then $(x_j,z_j)$ is parallel to $(x_1,z_1)$.  Assume (without loss of generality) that $j=2$ and  $(x_1,z_1)=\lambda (x_2,z_2).$ We study (\ref{eq:matr}),  with respect to $x_2$, under the assumption that  $(x_1,z_1)=\lambda (x_2,z_2)$ and continue this algorithm inductively (we will show the general step below).

We must also consider the case when $z_j=0$, for all $j \ge 2$. 
In this case (\ref{eq:matr}) becomes
$$
\left(|z_1|\sum_{1 < j < k \le M}  \left |\det\begin{pmatrix}
 x_j & x_k \\
 y_j & y_k 
\end{pmatrix}\right|\right) |z_1|  \le \left(|z_1|\sum_{ j=2}^M  | y_j|\right)\left(|z_1|\sum_{ j=2}^ M  |
 x_j |\right)
$$
or
$$
\left(\sum_{1 < j < k \le M}  \left |\det\begin{pmatrix}
 x_j & x_k \\
 y_j & y_k 
\end{pmatrix}\right|\right)   \le \left(\sum_{ j=2}^M  | y_j|\right)\left(\sum_{ j=2}^ M  |
 x_j |\right),
$$
which follows immediately from 
\begin{equation}\label{eq:class}
\left |\det\begin{pmatrix}
 x_j & x_k \\
 y_j & y_k 
\end{pmatrix}\right| \le |x_jy_k|+|x_ky_j|.
\end{equation}
  Continuing this process, we arrive to the case  
\begin{equation}\label{eq:asumpt}
(x_1, z_1)=\lambda_1(x_m, z_m), \dots, (x_{m-1}, z_{m-1})=\lambda_{m-1}(x_m, z_m),
\end{equation}
for some $2\le m \le M$. We  also denote $\lambda_m=1$ and we study $f$ and $g$ as functions of $x_m$.

Again, our first step is to confirm (\ref{eq:matr}), when $x_m \to \pm\infty.$ To do so, let us see how the functions $f$ and $g$ changed under (\ref{eq:asumpt}). Let us first consider the terms appearing in $g$:
\begin{equation*}
\left |\det\begin{pmatrix}
x_i & x_j & x_k \\
y_i & y_j & y_k \\
z_i & z_j & z_k
\end{pmatrix}\right|,
\end{equation*}
when $i<j<k$. For $m<i$,  the above determinant doesn't depend on $x_m$, so we only consider the case when $m\ge i$. 
\begin{itemize}
    \item If $m\ge k$, then the determinant is zero.
    \item if $i\le m<j$ then, when $|x_m|\to \infty$
\begin{equation*}
\left |\det\begin{pmatrix}
x_i & x_j & x_k \\
y_i & y_j & y_k \\
z_i & z_j & z_k
\end{pmatrix}\right| \sim |x_i| \left |\det\begin{pmatrix}
 y_j & y_k \\
 z_j & z_k
\end{pmatrix}\right| = |\lambda_i| |x_m| \left |\det\begin{pmatrix}
 y_j & y_k \\
 z_j & z_k
\end{pmatrix}\right|
\end{equation*} 
\item $j \le m<k$ then, when $|x_m|\to \infty,$ observing that $\lambda_i=z_i/z_m$ and $\lambda_j=z_j/z_m$ we get
\begin{equation*}
\left |\det\!\!\begin{pmatrix}
x_i & x_j & x_k \\
y_i & y_j & y_k \\
z_i & z_j & z_k
\end{pmatrix}\!\!\right| \!\sim \! \left| \lambda_i x_m \det\!\!\begin{pmatrix}
 y_j & y_k \\
 z_j & z_k
\end{pmatrix} - \lambda_j x_m \det\!\!\begin{pmatrix}
 y_i & y_k \\
 z_i & z_k
\end{pmatrix}\right|= |x_m| \frac{|z_k|}{|z_m|} \left|  \det\!\!\begin{pmatrix}
 y_i & y_j \\
 z_i & z_j
\end{pmatrix}\right|.
\end{equation*} 
 \end{itemize}
We also need to compute the behaviour of the terms appearing in $f$:
$$
\left |\det\begin{pmatrix}
x_i & x_j  \\
z_i & z_j 
\end{pmatrix}\right|,
$$
for $i<j$. When  $m\ge j$,  it is zero and when $m<i$,  it does not dependent on $x_m$. So we assume that $i \le m<j$. We get
$$
\left |\det\begin{pmatrix}
x_i & x_j  \\
z_i & z_j 
\end{pmatrix}\right| \sim |\lambda_i|\,|x_m|\,|z_j|.
$$
Thus to show that the  (\ref{eq:matr}) is true, as $|x_m|\to \infty$, we need to prove that
$$
\left[ (\sum_{i \le m} |\lambda_i|) \sum_{m<j<k<M}  \left |\det\begin{pmatrix}
 y_j & y_k \\
 z_j & z_k
\end{pmatrix}\right|  + \frac{1}{|z_m|} \left( \sum_{m <k} |z_k| \right) \!\!\!\left(\sum_{1\le i <j \le m} \left|  \det\begin{pmatrix}
 y_i & y_j \\
 z_i & z_j
\end{pmatrix}\right| \right) \right]\!\!\! \left(\sum_{i=1}^M |z_i| \right)
$$
$$
\le \left(\sum_{1 \le i < j  \le M} \left |\det\begin{pmatrix}
y_i & y_j \\
z_i & z_j 
\end{pmatrix}\right|\right) \left(\sum_{i \le m} |\lambda_i| \right)  \sum_{j>k}|z_k|.
$$
Multiplying both sides by $|z_m|$, we are reduced to
$$
\left[ (\sum_{i \le m} |z_i| ) \sum_{m<j<k<M}  \left |\det\begin{pmatrix}
 y_j & y_k \\
 z_j & z_k
\end{pmatrix}\right|  + \left( \sum_{m <k} |z_k| \right) \left(\sum_{1\le i <j \le m} \left|  \det\begin{pmatrix}
 y_i & y_j \\
 z_i & z_j
\end{pmatrix}\right| \right) \right] \left(\sum_{i=1}^M |z_i| \right)
$$
$$
\le \left(\sum_{1 \le i < j  \le M} \left |\det\begin{pmatrix}
y_i & y_j \\
z_i & z_j 
\end{pmatrix}\right|\right) \left(\sum_{i \le m} |z_i| \right)  \sum_{j>k}|z_k|
$$
Let $A'=\sum_{i=m}^M[0, (y_i, z_i)]$ and $B'=\sum_{i=1}^m[0, (y_i, z_i)]$, then the above becomes:
$$
\left[ |P_{e_1^\perp}B'| |A'|  + |P_{e_1^\perp}A'| |B'| \right] |P_{e_1^\perp}(A'+B')|
\le |A'+B'||P_{e_1^\perp}B'|  |P_{e_1^\perp}A'|.
$$
This is Bonnesen's inequality (\ref{eq:bonnesen}) in $\RL^2$
$$
 \frac{|A'|_2}{|P_{e_1^\bot}A|_{1}}+\frac{|B'|_2}{|P_{e_1^\bot} B'|_{1}} \le \frac{|A'+B'|_2}{|P_{e_1^\bot} (A'+B')|_{1}}.
$$
Our next step (if $m<M$) is to consider the critical points of $f$, as a function of $x_m$; they are given by the equations
$$
\det\begin{pmatrix}
x_m & x_j  \\
z_m & z_j 
\end{pmatrix} =0,
$$
for $z_j \not=0,$ $j \ge m+1$, if such $z_j$ exists and repeat the process for all $m <M$. If $z_j=0,$ for all $j\ge m+1$, let us confirm the inequality directly (as for the case $m=1$). To calculate
$$
\left|\det\begin{pmatrix}
x_i & x_j & x_k \\
y_i & y_j & y_k \\
z_i & z_j & z_k
\end{pmatrix}\right|
$$
we may consider cases: if $i,j,k \le m$, then the rank of the above matrix is at most to $2$ and the determinant is zero; if $i,j,k \le m$, the matrix has a row of zeros and the determinant is again zero. In the case when 
$i\le m <j<k$, we get
$$
\left|\det\begin{pmatrix}
x_i & x_j & x_k \\
y_i & y_j & y_k \\
z_i & z_j & z_k
\end{pmatrix}\right| =|z_i|\left|\det\begin{pmatrix}
x_j & x_k \\
y_j & y_k 
\end{pmatrix}\right|.
$$
When  $i< j\le m<k$, we use that 
$(x_i, x_j)$ is parallel to $(z_i, z_j)$ to get that
$$
\left|\det\begin{pmatrix}
x_i & x_j & x_k \\
y_i & y_j & y_k \\
z_i & z_j & z_k
\end{pmatrix}\right| =|x_k|\left|\det\begin{pmatrix}
 y_i & y_j \\
z_i & z_j
\end{pmatrix}\right|.
$$
We make a similar analysis on the right hand side of 
(\ref{eq:matr}) which  becomes
\begin{align*}
&\left(\sum_{i=1}^m |z_i| \right) \left[ \left(\sum_{i=1}^m |z_i| \right) \sum_{m<j<k}\left|\det\begin{pmatrix}
 x_j & x_k \\
y_j & y_k 
\end{pmatrix}\right|+ \sum\limits_{k>m} |x_k| \sum_{i<j\le m} \left|\det\begin{pmatrix}
 y_i & y_j \\
z_i & z_j
\end{pmatrix}\right|
\right]\\
& \le 
\left[\left(\sum_{i=1}^m |z_i| \right) 
\left(\sum_{j>m} |y_j|\right)+  \sum_{i<j\le m} \left|\det\begin{pmatrix}
 y_i & y_j \\
z_i & z_j
\end{pmatrix}\right|
\right]  \left[\left(\sum_{i=1}^m |z_i| \right) 
\left(\sum_{j>m} |x_j|\right)\right]
\end{align*}
The above inequality follows directly by simplification and application of \eqref{eq:class}.

We repeat the above process until we have $m=M$, thus  $(x_i, z_i)=\lambda_{i}(x_M, z_M)$ for each $i=1,\dots,M-1$. Thus, the left hand side of the inequality  (\ref{eq:matr}) is equal to zero. Indeed, each of the following matrices
 $$
\begin{pmatrix}
x_i & x_j & x_k \\
y_i & y_j & y_k \\
z_i & z_j & z_k
\end{pmatrix}
$$
has rank less or equal then $2$.
\end{proof}

\subsection{More equivalent formulations for zonoids}
\label{ss:equiv-z}

In this section, we show some additional equivalences (which thus hold in dimension $3$).

\begin{thm}\label{thm:zon}
Let $n \in {\mathbb N}$, then the following are equivalent.
\begin{enumerate}
\item For every $k\ge n$ and  for every family of vectors $u_1, \dots, u_k$ in $\RL^n$ the function $f:2^{[k]}\to\RL$  defined, for $S\subset[k]$,  by 
$$f(S)=\log\left|\sum_{i\in S}[0,u_i]\right|=\log\sum_{I\subset S,|I|=n}|\det(\{u_i\}_{i\in I})|
$$
is submodular: for all zonoids $A,B,C$ one has
\[
|A||A+B+C|\le|A+B||A+C|.
\]
\item For every  $k\ge n$ and for every family of vectors $u_1\dots, u_k$ in $\RL^n$, for every $u,v\in\RL^n$
\begin{align*}
\sum_{|I|=n}|\det(\{u_i\}_{i\in I})| &\sum_{|I|=n-2}|\det(u,v,(u_i)_{i\in I})|\\& \le \sum_{|I|=n-1}|\det(u,(u_i)_{i\in I})|\sum_{|I|=n-1}|\det(v,(u_i)_{i\in I})|.
\end{align*}
\item For  every zonoid $A$ in $\RL^n$ and every subspaces $E,F$ of $\RL^n$ such that $E^\perp \subseteq F$ we have
$$
|A| |P_{E \cap F} A| \le |P_{E}A| |P_{F} A|.
$$
\item For every $m=1,\dots, n$, for every zonoid $A$ and every orthonormal sequence  $u_1,\dots,u_m$, one has
\[
|A|^{m-1}\left|P_{[u_1,\dots,u_m]^\bot}A\right|_{n-m}\le\prod_{i=1}^m\left|P_{u_i^\bot}A\right|_{n-1}.
\]
\item For every $m=1,\dots, n$ and all zonoids $A, B_1,\dots,B_m$ in $\RL^n$, one has 
\[
|A|^{m-1}V(A[n-m],B_1,\dots, B_m)\le\frac{n^m(n-m)!}{n!}\prod_{i=1}^mV(A[n-1],B_i).
\]
\end{enumerate}
\end{thm}

\begin{proof}
The proof is based on translations of the properties described in Theorems \ref{thm:logsubmod-equiv} and \ref{th:linear} to the properties of zonoids.

We first show that 1. in Theorem \ref{thm:zon} is equivalent to 1. in Theorem \ref{thm:logsubmod-equiv} in the class of zonotopes. Indeed,  $f$ is submodular if and only if for every $S,T\subset \{1, \dots, k\}$
$$
\left|\sum_{i\in S \cup T}[0,u_i]\right| \left|\sum_{i\in S\cap T}[0,u_i]\right| \le  \left|\sum_{i\in S }[0,u_i]\right| \left|\sum_{i\in  T}[0,u_i]\right|,
$$
which is 1. in Theorem \ref{thm:zon} for zonotopes $A=\sum_{i\in S \cap T}[0,u_i]$, $B_1=\sum_{i\in S \setminus  T}[0,u_i]$ and $B_2=\sum_{i\in T \setminus  S}[0,u_i]$. We use  (\ref{form-vol-zo}) to finish the proof.

We next note that 2. in Theorem \ref{thm:zon} is equivalent to 3. of Theorem \ref{th:linear} in the case of zonotopes. Assume by homogeneity that $u,v\in S^{n-1}$. Then we  apply (\ref{form-vol-zo}) to get:
$$
\left|\sum_{i=1}^m[0,u_i]\right|=\sum_{I\subset [m],|I|=n}|\det(u_i)_{i\in I}| \mbox{ and }
\left|P_{v^\perp} \left(\sum_{i=1}^m[0,u_i] \right)\right|=\sum_{I\subset [m],|I|=n-1}|\det(v,(u_i)_{i\in I})|
$$
and the similar formula for  the volume of $P_{[u,v]^\perp} \left(\sum_{i=1}^m[0,u_i]\right)$.

 
Next, we  show that 3. is equivalent to 3. from Theorem \ref{th:linear}, which can be restated as 
\be\label{eq:z2}
|A||P_{[e_i,e_j]^\bot}A|_{n-2} \le   |P_{e_i^\bot}A|_{n-1}|P_{e_j^\bot}A|_{n-1},
\ee
for any zonoid $A$  and $i\not =j$, where $(e_1,\dots, e_n)$ is any orthonormal basis. Moreover 3. from Theorem \ref{thm:zon} is equivalent to
\be\label{eq:zn}
|A| |P_{[e_1,\dots e_k]^\perp} A|_{n-k} \le |P_{[e_1, \dots, e_i]^\perp} A|_{n-i} |P_{[e_{i+1},\dots, e_k]^\perp} A|_{n-(k-i)}.
\ee
for any $i < k \le n$ and any zonoid $A$ in $\RL^n$. 
Thus   (\ref{eq:z2}) is  a particular case of
(\ref{eq:zn}). To prove the reverse, we first notice that if   $(\ref{eq:z2})$   holds for any zonoid $A$ in $\RL^n$, then it also must hold for zonoids in any dimension $m \le n$. Indeed, for any zonoid $A$ 
in $\RL^m$, the cylinder $A \times [0,1]^{n-m}$ is a zonoid. Next, we may prove property (\ref{eq:zn}) by induction. Indeed, using (\ref{eq:z2}), it is true for $k=2$ and $i=1$, any $n \in \N$ and any zonoid $A$ in $\RL^n$. Assume the statement is true for some $k \in \N$ any $i < k \le n$. Let us apply the statement to the zonoid 
$P_{e_{k+1}^\perp} A$ to get 
\begin{align}
|P_{e_{k+1}^\perp} A|_{n-1} |P_{[e_1,\dots, e_k, e_{k+1}]^\perp} &A|_{n-(k+1)} \label{eq:ind1}\\
&\le |P_{[e_1, \dots, e_i, e_{k+1}]^\perp} A|_{n-(i+1)} |P_{[e_{i+1},\dots, e_k, e_{k+1}]^\perp} A|_{n-(k+1-i)}. \nonumber 
\end{align}
In addition, we  apply the inductive hypothesis to  $A$ and the subspace spanned by 
$\{e_1,\dots, e_i, e_{k+1}\}$ and the subspace spanned by $e_{k+1}$ to get 
\be\label{eq:ind2}
|A| |P_{[e_1,\dots, e_i, e_{k+1}]^\perp} A|_{n-(i+1)} \le |P_{[e_1, \dots, e_i]^\perp} A|_{n-i} |P_{e_{k+1}^\perp} A|_{n-1}.
\ee
Finally, we multiply (\ref{eq:ind1}) and (\ref{eq:ind2}) to finish the proof.

To prove that 3. implies 4., we apply (\ref{eq:zn}) to $k=m$ and $i=1$, we get that 
\[
\frac{|P_{[u_1,\dots, u_m]^\perp}A|}{|P_{[u_2,\dots, u_m]^\perp}A|}\le 
\frac{|P_{u_1^\perp}A|}{|A|}.
\]
In the same way, for every $1\le i\le m$, one has 
\[
\frac{|P_{[u_i,\dots, u_m]^\perp}A|}{|P_{[u_{i+1},\dots, u_m]^\perp}A|}\le 
\frac{|P_{u_i^\perp}A|}{|A|}.
\]
Taking the product of these inequalities, we get the result.

To prove that 4. implies 5., we use (\ref{eq:projvec}).
Thus (4) gives that for every orthonormal family of vectors $u_1,\dots, u_m$, one has 
\[
|A|^{m-1}V(A[n-m],[0,u_1],\dots, [0,u_m])
\le\frac{n^m(n-m)!}{n!}\prod_{i=1}^mV(A[n-1],[0,u_i]).
\]
Since this inequality is invariant with respect to any linear image of $A$ by an invertible map, it holds also for any independent $u_1,\dots, u_m$. Then, we deduce that 5. holds for any sums of segments. The inequality for zonoids follows by taking limits.

Finally, 5. with $m=2$ is equivalent to 5. in Theorem \ref{th:linear}. 
\end{proof}
Notice that 5. of Theorem \ref{thm:zon} can be rephrased by saying that, for any zonoid $A$ in $\RL^n$, the function $f:[n]\to\RL$ defined by $f(S)=\log(|P_{[e_i;i\in S]}A|)$ is submodular.

The inequalities analogous to 3. in Theorems  \ref{th:linear}  and \ref{thm:zon} belong to the class of local Loomis-Whitney type inequalities and were studied in many  works, including \cite{GHP02, FGM03, SZ16, AFO14, AAGHV17}, for the most general classes of convex bodies. In the next lemma, we present a new proof of a result from  \cite{GHP02}, the  proof uses the approach of the proof of Theorem 
\ref{th:linear}.
\begin{lem}\label{lem:GHPnew} Consider a convex body $K$ in $\RL^n$ and a pair of orthogonal vectors $u,v \in S^{n-1}$, then
$$
|K||P_{[u,v]^\perp} K| \le \frac{2(n-1)}{n}  |P_{u^\perp} K||P_{v^\perp} K|.
$$
\end{lem}
\begin{proof}
Let $L=[0,u]+\alpha [0,v],$ where $\alpha =|P_{u^\perp} K|/|P_{v^\perp} K|,$ noticing that the case $|P_{v^\perp} K|=0$ is trivial. 
Then 
$$n V(K[n-1], L)= |P_{u^\perp}K|_{n-1}+\alpha|P_{v^\perp}K|_{n-1} \mbox{ and } \frac{n(n-1)}{2}V(K[n-2], L[2])= \alpha|P_{[u,v]^\perp} K|.
$$
Using Minkowski's second inequality (\ref{M2}), we get
$$
 \left(\frac{|P_{u^\perp} K|+ \alpha |P_{v^\perp} K|}{n}\right)^2 \ge \frac{2 \alpha}{n(n-1)}  |K||P_{[u,v]^\perp} K|
$$
we substitute $\alpha =|P_{u^\perp} K|/|P_{v^\perp} K|$ to finish the proof.
\end{proof}

The main tool in the proof of Lemma \ref {lem:GHPnew} is Minkowski's second inequality, which relies on the fact that the polynomial $Q(t)=|K+tL|$ raised to the power $1/n$ is a concave function for $t\ge 0$.
We conjecture that the concavity properties of this polynomial can be improved for $Z$ being a zonoid and $L$ being a finite sum of $m\le n$ segments. More precisely, we conjecture that for any $m\le n$ and any sequence of vectors $u_1, \dots, u_m$, $m \le n$   from $\RL^n$ if  $P(t)=|Z + t \sum_{i=1}^m [0,u_i]|,$ then $P^{1/m}(t)$  is  a concave function for $t \ge 0$. This conjecture would follow if the statements of Theorem  \ref{thm:zon} or Theorem \ref{th:linear} would be true. Still we can prove the following proposition in $\RL^3$.
\begin{prop}\label{propconc} Let $Z$ be a zonoid in $\RL^3$ and 
$u,v$  be two  vectors from $\RL^3$. Let  $P(t)=|Z + t ([0,u]+[0,v])|,$ then $P^{1/2}(t)$  is  a concave function for $t \ge 0$.
\end{prop}
\begin{proof} We may assume that the vectors $u, v$ are linearly independent. Then,
applying a linear transformation $T$ to $Z$ and $u, v$ we may assume that vectors $u,v$ are orthogonal to each others and belong to $S^{2}.$ Using that
$Z + t ([0,u]+[0,v])$ is again a zonoid, it is enough for us to show that  $(P^{1/2}(0))'' \le 0$, or $2P(0)P''(0)\le P'(0)^2$. Using that 
\[
P(0)=|Z|,\quad P'(0)=|P_{u^\perp} Z|+ |P_{v^\perp} Z|\quad \hbox{and}\quad P''(0)=2|P_{[u,v]^\bot}Z|.
\]
and  Theorem \ref{thm:hope}, we get 
$$
2P(0)P''(0)=4 |Z||P_{[u,v]^\bot}Z|\le 4 |P_{u^\bot}Z||P_{v^\bot}Z| \le  (|P_{u^\perp} K|+ |P_{v^\perp} K|)^2=P'(0)^2.
$$
\end{proof}

It follows from  part 6. of Theorem \ref{th:linear} and Theorem \ref{thm:hope}  that if  $P(t)=|Z+t([0,u]+[0,v])|$, then $P(t)$, as a polynomial on $\RL$, has only real roots, for any zonoid  $Z$ and any $u,v\in\RL^3$. Thus another way to prove Proposition  \ref{propconc} is to notice the following simple property. Consider a quadratic polynomial $P$ with positive coefficients, then $\sqrt{P(t)}$ is concave for $t\ge 0$ if and only if $P$ has only real roots. Thus, to study concavity of $\sqrt{P_C(t)}=|Z+tC|^{1/2},$ when $C$ is two dimensional, one may study the roots of $P$. The subject of roots of Steiner-type polynomials has attracted a fair bit of attention in the literature, see, e.g., \cite{HHS12} and references therein. Proposition \ref{propconc} makes the following conjecture plausible: 
\begin{ques}\label{2dimbezout}
Let $Z$ be a zonoid in $\RL^n$. Then is it true that the polynomial $P_C(t)=|Z+tC|$ has only real roots for every convex body $C$ of dimension $2$?
\end{ques}
%

This question is directly connected to a question of Adam Marcus that we learned from Guillaume Aubrun about the roots of Steiner polynomials of zonoids,
and, in fact, it was one of the starting point of this  investigation. 

\begin{ques}\label{marcus}
Let $Z$ be a zonoid in $\RL^n$. Then is it true that the Steiner polynomial $P_Z(t)=|Z+tB_2^n|$ has only real roots?
\end{ques}

Observe that, in the plane, even more is true: for any convex bodies $K,L$, the polynomial $P_{K,L}(t)=|K+tL|$ has only real roots.  Indeed, $\sqrt{P_{K,L}(t)}$ is concave for $t\ge0$ by Brunn-Minkowski inequality. This can be also seen from the computing the discriminant and noticing that $V(K,L)^2-|K||L|\ge0$. 

Analogously,  Question \ref{2dimbezout} is equivalent to the following question for mixed volumes: fix $n\ge 3$, and let
$K$  a zonoid in $\RL^n$ and $L$ be a two dimensional zonoid, is it true that
$$
|K|V(K[n-2],L[2]) \le \frac{n}{2(n-1)}V(K[n-1], L)^2?
$$
The above inequality is true for $L$ being a parallelogram, as follows from Theorem \ref{thm:logsubmod-equiv}. But, it is not true for general zonoids, as we  show in the following proposition inspired by a work by Victor Katsnelson \cite{Kat07}. 

\begin{prop}\label{ref:contrm}
Let $n\ge3$. Then there exists a zonoid $Z$ in $\RL^n$ such that the Steiner polynomial $P_Z(t)=|Z+tB_2^n|$ and the polynomial $Q_Z(t)=|tZ+B_2^n|$ has  roots which are not real. 
\end{prop}
\begin{proof}
Noticing that $Q_Z(t)=t^nP_Z(1/t),$ it is enough to show that $P_Z(t)$ has a non-real root.
Consider $Z=B_2^{2}\times\{0\}\subset\RL^2\times\RL^{n-2}$. Integrating on sections parallel to $\RL^{2}\times\{0\}$ and changing variable, we get that, for $t\ge0$,
\[
P_Z(t)=|Z+tB_2^n|= |B_2^2|\int_{tB_2^{n-2}}\left(1+\sqrt{t^2-|x|^2}\right)^{2}dx=t^{n-2} |B_2^{2}|\int_{B_2^{n-2}}\left(1+t\sqrt{1-|x|^2}\right)^{2}dx.
\]
 Since this last expression is the restriction of a polynomial to $t\ge0$, the equality between $P_Z(t)$ and the last term is valid on whole $\RL$. Let $t\in\RL$ be fixed. For almost all $x\in B_2^{n-2}$, $(1+t\sqrt{1-|x|^2})^{2}>0$, thus $P_Z(t)>0$, for $t\neq0$, so the only real root of $P_Z$ is $0$
and it is of order $n-2$. Hence $P_Z$ has exactly $2$ non-real roots. 
\end{proof}

Proposition~\ref{ref:contrm} shows that the answers to both Question \ref{2dimbezout}
and Question \ref{marcus} are negative.

Note that the example of the zonoid $Z$ created in Proposition \ref{ref:contrm} is flat (i.e.,
has an affine hull of lower dimension than the ambient space), but one can 
replace it by a non-flat zonoid by using a small perturbation and the continuity of the roots of polynomials.

\section{Inequalities for $L_p$-zonoids}
\label{sec:Lp-z}

Firey \cite{Fir62} extended the concept of Minkowski sum and introduced, for each real $p \ge 1$, a $p$-linear combination
of convex bodies, the so-called $\ell_p$-sum
 $K\oplus_p L$ of the convex bodies $K$ and $L$ containing the origin by: 
\[
h_{ K\oplus_p L}(x)=\left(h_K(x)^p+h_L(x)^p\right)^\frac{1}{p},\quad \forall x\in\RL^n.
\]
For any linear transform $T$ on $\RL^n$, one has $T(K\oplus_p L)=(TK)\oplus_p (TL)$. In a series  of papers, Lutwak \cite{Lut93:1, Lut96} showed that the  Firey sums lead to a
Brunn-Minkowski theory for each $p \ge 1$, including $L_p$-Brunn-Minkowski inequality, definition and inequalities for $L_p$-mixed volumes, $L_p$-Minkowski problem, as well as many other applications. In this section, we  show the connection of the discussion from previous sections to this theory. 

An $L_p$-zonotope is the $\ell_p$-sum of centered segments and an $L_p$-zonoid is the Hausdorff limit of $L_p$-zonotopes. For $p=2$, an $L_2$-zonoid is always an ellipsoid, possibly living in a lower-dimensional subspace, thus it can be written as the sum of $m\le n$ orthogonal segments and is therefore an $L_2$-zonotope. 
The following extension of Conjecture \ref{constrong} is thus natural. 
\begin{ques}\label{queslpzonoids}
Let $p\ge1$ and consider $L_p$-zonoids $A,B$ in $\RL^n$ is it true that
\begin{equation}\label{lp-sumconj}
\left(\frac{|A\oplus_pB|}{|P_{u^\bot}(A\oplus_pB)|_{n-1}}\right)^p\geq \left(\frac{|A|}{|P_{u^\bot}A|_{n-1}}\right)^p+\left(\frac{|B|}{|P_{u^\bot} B|_{n-1}}\right)^p?
\end{equation}
\end{ques}

\subsection{The case $p=2$}
The next theorem gives an affirmative answer to this question in the case $p=2$. 
\begin{thm}\label{thm:Zonoidsstrong2}
Let $A, B$ be a pair of $L_2$-zonoids in ${\mathbb R}^n$ and let $u$ in $S^{n-1}$. Then
$$\left(\frac{|A\oplus_2B|}{|P_{u^\bot}(A\oplus_2B)|_{n-1}}\right)^2\geq \left(\frac{|A|}{|P_{u^\bot}A|_{n-1}}\right)^2+\left(\frac{|B|}{|P_{u^\bot} B|_{n-1}}\right)^2,$$
with equality if and only if $A$ and $B$ have parallel tangent hyperplanes at $\rho_A(u)u$ and $\rho_B(u)u$.
\end{thm}
\noindent We first give two proofs of the inequality and then prove the equality case.
 
 \noindent{\it Proof 1:} For $m\ge n$ and $u_1,\dots,u_m\in\RL^n$, let $U$ be the $n\times m$ matrix whose columns are $u_1,\dots,u_m$. One has $[-u_1,u_1]\oplus_2\cdots\oplus_2[-u_m,u_m]=UB_2^m=\sqrt{UU^*}B_2^n$. Indeed,
\begin{align*}
h^2_{[-u_1,u_1]\oplus_2\cdots\oplus_2[-u_m,u_m]}(x) &=\sum_{i=1}^m \langle u_i, x \rangle^2=
\sum_{i=1}^m \langle U e_i, x\rangle^2 = \sum_{i=1}^m \langle e_i, U^* x\rangle^2=|U^*x|^2 = h^2_{U B_2^m}(x)\\&
=\langle U^*x, U^*x\rangle=\langle\sqrt{UU^*}x, \sqrt{UU^*}x\rangle=|\sqrt{UU^*}x|^2 =h^2_{\sqrt{UU^*}B_2^n}(x).
\end{align*}
Using this and the Cauchy-Binet formula, one has 
\begin{equation}\label{eq:volform}
|[-u_1,u_1]\oplus_2\cdots\oplus_2[-u_m,u_m]|^2=|U B_2^m|^2_n=  \det (UU^*) |B_2^n|^2 =|B_2^n|^2\sum_{|I|=n}(\det(u_i)_{i\in I})^2 .
\end{equation}
Next, using  
$$
P_{u^\bot}([-u_1,u_1]\oplus_2\cdots\oplus_2[-u_m,u_m])=[-P_{u^\bot}u_1,P_{u^\bot}u_1]\oplus_2\cdots\oplus_2[-P_{u^\bot}u_m,P_{u^\bot}u_m], $$
we get that
\[
|P_{u^\bot}([-u_1,u_1]\oplus_2\cdots\oplus_2[-u_m,u_m])|^2=|B_2^{n-1}|^2\sum_{|I|=n-1}(\det(P_{u^\bot}u_i)_{i\in I})^2.
\]
Thus our goal is to prove that, for $N\ge N'$ and $A=[-u_1, u_1]\oplus_2\cdots\oplus_2[-u_{N'},u_{N'}]$ and  $B=[-u_{N'+1}, u_{N'+1}]\oplus_2\cdots\oplus_2[-u_{N},u_{N}]$ we have
\begin{align}
 \frac{\sum\limits_{M\subset\{1,\dots,N'\}, |M|=n}|\det(\{u_{m}\}_{m \in M}|^2} {\sum\limits_{M\subset\{1,\dots,N'\},|M|= n-1}|\det(\{P_{e_1^\perp}u_{m}\}_{m \in M})|^2}& +  
\frac{\sum\limits_{L\subset\{N'+1,\dots,N\},|L|= n}|\det(\{u_{l}\}_{l \in L}|^2} {\sum\limits_{L\subset\{N'+1,\dots,N\}, |L|= n-1}|\det(\{P_{e_1^\perp}u_{l}\}_{l \in L}|^2} \nonumber \\
&\le \frac{\sum\limits_{M\subset\{1,\dots,N\},|M|= n}|\det(\{u_{m}\}_{m\in M}|^2}{\sum\limits_{M\subset\{1,\dots,N\}, |M|= n-1}|\det(\{P_{e_1^\perp}(u_{m})\}_{m\in M}|^2}.
\label{eq:determinantsstrong2}
\end{align}
For $I\subset \{1,\dots,m\}$ let  $U_I$ be the $n\times n$ submatrix built from $U$ by taking the columns with indices in $I$.
Denote by $z_1,\dots,z_n\in\RL^m$ the rows of the matrix $U$. Since $n\le m$, the set $\sum_{i=1}^n[0,z_i]$ is a parallelotope leaving in a $n$-dimensional subspace of $\RL^m$. Thus, 
\[
\left|\sum_{i=1}^n[0,z_i]\right|_n^2=\sum_{|I|=n}(\det\  U_I)^2.
\]
Therefore, we get that, for $m\ge n$, an $n\times m$ matrix $U$ whose rows are $z_1,\dots,z_n\in\RL^m$ and columns $u_1,\dots,u_m\in\RL^n$:
\[
|[-u_1,u_1]\oplus_2\cdots\oplus_2[-u_m,u_m]|_n=|B_2^n|\left|\sum_{i=1}^n[0,z_i]\right|_n.
\]
Let us write $U$ as a block matrix with two blocks   $U=(U_1|0)+(0|U_2)$, with $U_1$ being an $n\times k$ matrix and $U_2$ an $n\times (m-k)$ matrix and we denote $V=(U_1|0)$ and $W=(0|U_2)$. Moreover, denote by $U', V', W'$ the matrices obtained from $U,V,W$  by erasing the $n^{th}$ row. 
Then we only need to prove that 
\[
\frac{\sum_{|I|=n}(\det\  U_I)^2}{\sum_{|I|=n-1}(\det\  U'_I)^2}\ge
\frac{\sum_{|I|=n}(\det\  V_I)^2}{\sum_{|I|=n-1}(\det\  V'_I)^2}+
\frac{\sum_{|I|=n}(\det\  W_I)^2}{\sum_{|I|=n-1}(\det\  W'_I)^2}.
\]
Recall that $z_1,\dots,z_n\in\RL^m$ are the rows of $U$. Thus the rows of $V$ are $P_Ez_1,\dots,P_Ez_n$, where $P_E$ denotes the projection on the first $k$ coordinates and the rows of $W$ are $P_{E^\bot}z_1,\dots,P_{E^\bot}z_n$, where $P_{E^\bot}$ denotes the projection on the $n-k$ last coordinates. Thus we only need to show the following relationship for low-dimensional parallelotopes
\[
\left(\frac{\left|\sum_{i=1}^n[0,z_i]\right|}{\left|\sum_{i=1}^{n-1}[0,z_i]\right|}\right)^2
\ge
\left(\frac{\left|\sum_{i=1}^n[0,P_Ez_i]\right|}{\left|\sum_{i=1}^{n-1}[0,P_Ez_i]\right|}\right)^2
+\left(\frac{\left|\sum_{i=1}^n[0,P_{E^\bot}z_i]\right|}{\left|\sum_{i=1}^{n-1}[0,P_{E^\bot}z_i]\right|}\right)^2.
\]
Using that the volume of a parallelotope is the product of the volume of one of its face and its height, we get that, if $H_n={\mathrm{span}}(z_1,\dots, z_{n-1})$, then
\[
\frac{\left|\sum_{i=1}^n[0,z_i]\right|}{\left|\sum_{i=1}^{n-1}[0,z_i]\right|}=d(z_n,H_n).
\]
So we are reduced to prove that 
\[
d(z_n,H_n)^2\ge d(P_Ez_n,P_EH_n)^2+d(P_{E^\bot}z_n,P_{E^\bot}H_n)^2.
\]
Let $h_n\in H_n$ such that $d(z_n,H_n)=|z_n-h_n|$. By Pythagoras' theorem, 
\[
d(z_n,H_n)^2=|z_n-h_n|^2=\|P_Ez_n-P_Eh_n\|^2+\|P_{E^\bot}z_n-P_{E^\bot}h_n\|^2.
\]
Since $P_Eh_n\in P_EH_n$ and $P_{E^\bot}z_n\in P_{E^\bot}H_n$, we conclude. \qedwhite

\noindent{\it Second proof of the inequality:} We give another proof of Theorem~\ref{thm:Zonoidsstrong2}, using the comparison of the $\ell_2$ sum and the radial $2$-sum. 
For any symmetric convex bodies $K$ and $L$ in $\RL^n$, one has
\begin{equation}\label{radial-2-sum}
K\oplus_2L\supset K \widetilde{+}_2L:=\cup_{u\in S^{n-1}}\left[0,\sqrt{\rho_K^2(u)+\rho_L^2(u)}u\right].
\end{equation}
Indeed, using support functions, the fact that $K\supset\rho_K(u)[-u,u]$ implies that $h_K(x)\ge\rho_K(u)|\langle x,u\rangle|$, for any $x\in\RL^n$ and thus 
\[
h_{K\oplus_2L}(x)=\sqrt{h_K(x)^2+h_L(x)^2}\ge \sqrt{\rho_K(u)^2+\rho_L(u)^2}|\langle x,u\rangle|=\sqrt{\rho_K(u)^2+\rho_L(u)^2}h_{[-u,u]}(x).
\]
The formula (\ref{radial-2-sum}) can be also restated in the language of radial functions:
\begin{equation}\label{radial-2-sum-simp}
\rho_{K\oplus_2L}(u) \ge \sqrt{\rho_K^2(u)+\rho_L^2(u)}, \mbox{ for all } u \in S^{n-1}.
\end{equation}
Next, we notice nice a formula for the volume of the orthogonal hyperplane projection of an ellipsoid (see, for example, \cite{CO84,Riv07}), to which we give a very simple proof. Let $\E=T B_2^n,$ for some positive definite $T$ then
\begin{equation}\label{zonoidproj}
\frac{|P_{u^\bot}\E|}{|\E|\,\|u\|_\E}=\frac{nV(TB_2^n[n-1], [0,u])}{|TB_2^n|\,|T^{-1} u|} =\frac{nV(B_2^n[n-1], [0,T^{-1}u])}{|B_2^n|\,|T^{-1} u|}=\frac{|B_2^{n-1}|}{|B_2^n|}.
\end{equation}
Using the above we get
\begin{equation}\label{ellipsoid-projection-formula-simp}
\frac{|\E|}{|P_u\E|}=\frac{|B_2^n|}{|B_2^{n-1}|}\rho_\E(u).
\end{equation}
Thus, using this formula and \eqref{radial-2-sum}, we deduce that for any ellipsoids $A, B$
\begin{align}\label{ellipses}
\frac{|A\oplus_2 B|^2}{|P_u(A\oplus_2 B)|^2}&=\frac{|B_2^n|^2}{|B_2^{n-1}|^2}\rho_{A\oplus_2B}(u)^2  \nonumber \\
&\ge\frac{|B_2^n|^2}{|B_2^{n-1}|^2}(\rho_{A}(u)^2+\rho_{B}(u)^2)=\frac{|A|^2}{|P_uA|^2}+\frac{|B|^2}{|P_uB|^2}.
\end{align}
\qedwhite

\noindent{\it Proof of the equality case:} 
This second proof  also helps us to treat the equality case.
From (\ref{ellipsoid-projection-formula-simp}) there is equality if and only if
$$ \rho_{A\oplus_2B}(u)^2= \rho_{A}(u)^2+\rho_{B}(u)^2.
$$
The above is equivalent to the $(\rho_{A}(u)^2+\rho_{B}(u)^2)^{1/2} u  \in \partial (A\oplus_2 B)$. From here we get that,  if $n$ is a normal vector to $\partial (A\oplus_2 B)$ at $(\rho_{A}(u)^2+\rho_{B}(u)^2)^{1/2} u $ then
$$
h_{A\oplus_2 B}(n)= (\rho_{A}(u)^2+\rho_{B}(u)^2)^{1/2} \langle u, n\rangle.
$$
The above is equivalent to 
$
h_{A}(n)= \rho_{A}(u) \langle u, n\rangle
\mbox{ and } h_{B}(n)= \rho_{B}(u) \langle u, n\rangle,$
or simply to say that the normal vector to $\partial A$ at $\rho_{A}(u)u$ is the normal vector to $\partial B$ at $\rho_{B}(u)u$. 
\qedwhite

\begin{rem} If in the proof of equality case in the Theorem \ref{thm:Zonoidsstrong2}  we would represent $A=T_1B_2^n$ and $B=T_2B_2^n,$ where $T_1, T_2$ are two positive symmetric matrices. Then $
\|x\|^2_{A}=\langle T_1^{-2} x, x \rangle.
$  and  the normal vector  to $\partial A$ at $\rho_{A}(u)u$ is parallel  to
$T_1^{-2} u$, and the similar statement is true for $B$. Thus our condition on parallel tangent hyperplanes is equivalent to the fact that there is $\lambda>0$ such that 
$
(T_1^{-2} - \lambda T_2^{-2}) u =0
$
or simply that $u$ is an eigenvector for matrix $T_1^{2} T_2^{-2}$. 
\end{rem}
\begin{rem} Observe that in Theorem \ref{thm:Zonoidsstrong2} there is an equality at least for $n$ directions $u$ (not counting $-u$) and there is equality for every $u$ if and only if two ellipsoids are homothetic.
\end{rem}

Let us now present some consequences or Theorem \ref{thm:Zonoidsstrong2}.

\begin{coro}\label{refcorol2}
Let $n$ be a positive integer. Let $A$ and $B$ be $L_2$-zonoids in $\RL^n$ and $u$ in $S^{n-1}$. Then, the function $h$ defined, for $t\ge0$, by 
\[
h_2(t)=\frac{|A\oplus_2(\sqrt{t}B)|^2}{|P_{u^\bot}(A\oplus_2(\sqrt{t}B))|^2_{n-1}}
\]
is concave on $\RL_+$.
\end{coro}

\begin{proof}
For any $\lambda\in[0,1]$, one has $A=(\sqrt{1-\lambda}A)\oplus_2(\sqrt{\lambda}A)$.
Thus one deduces that
\[
A\oplus_2\left(\sqrt{(1-\lambda) s+\lambda t}B\right)
=\left(\sqrt{1-\lambda}(A\oplus_2\sqrt{s}B)\right)\oplus_2
\left(\sqrt{\lambda}(A\oplus_2\sqrt{t}B)\right).
\]
Using Theorem \ref{thm:Zonoidsstrong2} and the  homogeneity of volume, we deduce that $h_2$ is concave.
\end{proof}

Next we show that Theorem \ref{thm:Zonoidsstrong2} has the following additional applications:

\begin{thm}\label{thm:strong-proj}
Let $n$ be an integer, then for any $1\le k\le n$ and for every pair of $L_2$-zonoids $A$ and $B$ in ${\mathbb R}^n$  and any ($n-k$)-dimensional subspace $E$ of $\RL^n$ one has
\begin{equation}\label{eq:DCTProj}
\left(\frac{|A\oplus_2B|}{|P_{E}(A\oplus_2B)|_{n-k}}\right)^\frac{2}{k}\geq \left(\frac{|A|}{|P_E A|_{n-k}}\right)^\frac{2}{k}+
\left(\frac{|B|}{|P_E B|_{n-k}}\right)^\frac{2}{k}.
\end{equation}
\end{thm}

\begin{proof} The proof goes by induction on $k$. Theorem \ref{thm:Zonoidsstrong2} establishes the case $k=1$ and any $n\ge1$. We assume that the inequality holds for some $1\le k\le n-1$ for all $L_2$-zonoids $A,B$ in $\RL^n$ and all $n-k$ dimensional subspace of $\RL^n$. Let $E$ be a $n-k-1$ dimensional subspace of $\RL^n$. Then one may write $E=F\cap u^\bot$, for some $n-k$ dimensional subspace $F$ and $u\in F^\bot$. Then, applying Theorem~\ref{thm:Zonoidsstrong2} to $P_FA$ and $P_FB$ and using that $P_{u^\bot}\circ P_F=P_E$ we get 
\begin{equation}\label{eq:DCTProj-ind}
\left(\frac{|P_F(A\oplus_2B)|_{n-k}}{|P_{E}(A \oplus_2 B)|_{n-k-1}}\right)^2\geq \left(\frac{|P_FA|_{n-k}}{|P_E A|_{n-k-1}}\right)^2+
\left(\frac{|P_FB|_{n-k}}{|P_E B|_{n-k-1}}\right)^2.
\end{equation}
Applying \eqref{eq:DCTProj} to $E=F$, raising the equation to power $k/2$ and taking the product with \eqref{eq:DCTProj-ind}, we get
\begin{align}
\label{eq:DCTProj-ind-prod}
&\frac{|A\oplus_2B|}{|P_E(A\oplus_2B)|_{n-k-1}}\geq \\\nonumber & \left(\left(\frac{|P_F A|_{n-k}}{|P_EA|_{n-k-1}}\right)^2 +
\left(\frac{|P_FB|_{n-k}}{|P_E B|_{n-k-1}}\right)^2\right)^{\frac{1}{2}}\left( \left(\frac{|A|}{|P_F A|_{n-k}}\right)^\frac{2}{k}+
\left(\frac{|B|}{|P_F B|_{n-k}}\right)^\frac{2}{k}\right)^{\frac{k}{2}}.
\end{align}
From H\"older's inequality, we conclude that
\[
\frac{|A\oplus_2B|}{|P_E(A\oplus_2B)|_{n-k-1}}\geq\left(\left(\frac{|A|}{|P_E A|_{n-k-1}}\right)^\frac{2}{k+1}+
\left(\frac{|B|}{|P_E B|_{n-k-1}}\right)^\frac{2}{k+1}\right)^{\frac{k+1}{2}},
\]
which is \eqref{eq:DCTProj} for $k+1$.
\end{proof}

\begin{cor}\label{cor:zon-strong}
Let $k,n$ be a  integer with $1\le k\le n.$ Then  for every $L_2$-zonoids $A$ and $B$   in $\RL^n$ and every zonoids $Z_1,\dots, Z_k$ in $\RL^n,$
\begin{align}\label{DCTmixed2}
&\left(\frac{|A \oplus_2 B|}{V((A \oplus_2 B)[n-k],Z_1, \dots, Z_k)}\right)^{\frac{2}{k}}\geq \\ \nonumber &\left(\frac{|A|}{V(A[n-k],Z_1, \dots, Z_k)}\right)^{\frac{2}{k}} + \left( \frac{|B|}{V(B[n-1],Z_1, \dots, Z_k)}\right)^{\frac{2}{k}},
\end{align}
and thus  the function $f$ defined, for $t\ge0$, by 
\[
f(t)=\frac{|A\oplus_2 \sqrt{t}B|^{\frac{2}{k}}}{|V((A \oplus_2 \sqrt{t}B)[n-k],Z_1, \dots, Z_k)|^{\frac{2}{k}}_{n-k}}
\]
is concave on $\RL_+$. Moreover
\begin{equation}\label{DCT2}
\frac{|A \oplus_2 B|^2}{|\partial (A \oplus_2 B)|^2}\geq \frac{|A|^2}{|\partial A|^2}+\frac{|B|^2}{|\partial B|^2},
\end{equation}
 and thus the function $g$ defined, for $t\ge0$, by 
\[
g(t)=\frac{|A \oplus_2 \sqrt{t}B|^2}{|\partial (A \oplus_2 \sqrt{t}B)|^2_{n-1}}
\]
is concave on $\RL_+$.
\end{cor}

\begin{proof}


First notice that Theorem \ref{thm:strong-proj} may be reformulated in the following way. Let $u_1,\dots, u_k$ be an orthonormal system in $\RL^n$. Then using (\ref{eq:projvec}) we have  
\begin{align}\label{eq:DCTPmixed-vectors}
\left(\frac{|A\oplus_2 B|}{V((A\oplus_2 B)[n-k],[0,u_1],\dots, [0,u_k])}\right)^\frac{2}{k} \geq& \\ \nonumber \left(\frac{|A|}{V(A[n-k],[0,u_1],\dots, [0,u_k])}\right)^\frac{2}{k}&+
\left(\frac{|B|}{V(B[n-k],[0,u_1],\dots, [0,u_k])}\right)^\frac{2}{k}.
\end{align}
Applying a linear transform,   \eqref{eq:DCTPmixed-vectors} holds for any linearly   independent system $u_1,\dots, u_k$. Then, for $x,y\ge0$,  define
\begin{equation}\label{phiconcave}
\f(x,y)=(x^{-\frac{2}{k}}+y^{-\frac{2}{k}})^{-\frac{k}{2}}=\|(x,y)\|_{-\frac{2}{k}}.
\end{equation}
 For a compact convex set $A$ and  $u_1, \dots, u_k\in\RL^n$, let 
\begin{equation}\label{def:psi}
\psi_A(u_1, \dots, u_k)=\frac{V(A[n-k],[0,u_1],\dots, ,[0,u_k])}{|A|}.
\end{equation} 
From (\ref{eq:DCTPmixed-vectors}) we know that, if $u_1, \dots, u_k$  are linerly independent, then 
$$
\psi_{A\oplus_2B}(u_1, \dots, u_k)\le\f(\psi_A(u_1, \dots, u_k),\psi_B(u_1, \dots, u_k)).
$$ 
For $i=1, \dots, k$, let  $Z_i=[0,u_{i,1}]\oplus_2\cdots\oplus_2[0,u_{i,m_i}]$ be a $2$-zonotope. Assume that for any set  of distinct $k$ vectors from the set $\{u_{i,j}\}$ is an independent sequence (this can be achieved by a small perturbation of vectors $u_{i,j}$). Using that $\f$, being a $-\frac{2}{k}$-norm, satisfies the reverse Minkowski inequality, we deduce that 
\begin{align*}
\sum_{i=1}^k\sum_{j_i=1}^{m_i}\psi_{A\oplus_2B}(u_{1, j_1},\dots, u_{k, j_k}) &\le\sum_{i=1}^k\sum_{j_i=1}^{m_i}\f(\psi_A(u_{1, j_1},\dots, u_{k, j_k}),\psi_B(u_{1, j_1},\dots, u_{k, j_k}))\\ &\le \f\left(\sum_{i=1}^k\sum_{j_i=1}^{m_i}\psi_A(u_{1, j_1},\dots, u_{k, j_k}), \sum_{i=1}^k\sum_{j_i=1}^{m_i}\psi_B(u_{1, j_1},\dots, u_{k, j_k})\right).
\end{align*}
Thus we get
\[
\frac{V((A\oplus_2B)[n-k],Z_1,\dots, Z_k)}{|A\oplus_2B|}\le \f\left(\frac{V(A[n-k],Z_1,\dots, Z_k)}{|A|},\frac{V(B[n-k],Z_1,\dots, Z_k)}{|B|}\right),
\]
which is (\ref{DCTmixed2}), for $Z_1,\dots, Z_k$ being a zonotopes. The result for zonoids follows by approximation.
We apply (\ref{DCTmixed2}) with $k=1$ and $Z_1=B_2^n$ to prove (\ref{DCT2}).
The concavity of $f$ and $g$ is  proved with the method used for Corollary~\ref{refcorol2}.

\end{proof}

\subsection{The case $p\neq2$}

\begin{prop} The answer to  Question \ref{queslpzonoids} is 
negative when $p>2$ and $n\ge 2$: a counterexample is given by $A=B_2^n$ and $B=\varepsilon^{\frac{1}{p}}[- v,  v]$, for some $\varepsilon>0$.
\end{prop}
 \begin{proof} 

Let $p>2$. We disprove the weaker statement
\begin{equation}\label{lp-sumconjw}
\frac{|A\oplus_pB|}{|P_{u^\bot}(A\oplus_pB)|_{n-1}} \geq \frac{|A|}{|P_{u^\bot}A|_{n-1}}.
\end{equation}

We restate (\ref{lp-sumconjw}) with  $B=\varepsilon^{\frac{1}{p}}[- v,  v],$ $v\in S^{n-1}$ and $\varepsilon \to 0.$ 
For this we use the extension of the classical notion of the mixed volumes  introduced by 
Lutwak \cite{Lut93:1, Lut96}, who proved that 
\begin{equation}\label{lut1}
 \lim\limits_{\varepsilon \to 0}\frac{|A\oplus_p(\varepsilon^{\frac{1}{p}} \, B)|- |A|}{\varepsilon}= \frac{1}{p} \int\limits_{S^{n-1}} h^p_B(u)h_A^{1-p} dS_A(u),
\end{equation}
for $p>1$  and  all convex compact sets  $A, B,$ containing the origin and defined
\begin{equation}\label{lut2}
V_p(A[n-1],B)=\frac{1}{n} \int\limits_{S^{n-1}} h^p_B(u)h_A^{1-p} dS_A(u).
\end{equation}
Taking $B=\varepsilon^{\frac{1}{p}}[- v,  v]$ for some $v\in S^{n-1},$ we get
$$
\Big|A\oplus_p B \Big| =|A|+ \frac{\varepsilon }{p}\int\limits_{S^{n-1}} |\langle v, x\rangle |^p h_A^{1-p} dS_A(x)+o(\varepsilon)
$$
Assuming $v$ in $u^\perp,$ we get $$
\Big|P_{u^\perp} A\oplus_p P_{u^\perp} B \Big| =|P_{u^\perp} A|+  \frac{\varepsilon }{p}\int\limits_{S^{n-1}\cap u^\perp} |\langle v, x\rangle |^p h_A^{1-p}(x) dS_{P_{u^\perp}A}(x)+o(\varepsilon).
$$
Assume, by way of contradiction, that   (\ref{lp-sumconjw}) is true for all $L_p$-zonoids $A\subset \RL^n.$ Then,
\begin{equation}\label{limp}
|A| \int\limits_{S^{n-1}\cap u^\perp} |\langle v, x\rangle |^p h_A^{1-p}(x) dS_{P_{u^\perp}A}(x) \le 
|P_{u^\perp}A| \int\limits_{S^{n-1}} |\langle v, x\rangle |^p h_A^{1-p}(x) dS_A(x).
\end{equation}
Now we take $A=B_2^n$, $u=e_2$ and $v=e_1$.  Notice that $B_2^n$ is an $L_p$ zonoid for all $p\ge 1$. Next (\ref{limp}) becomes
\begin{equation}\label{eqball}
|B_2^n| \int_{S^{n-2}} |x_1|^p  dS_{B_2^{n-1}}(x) \le 
|B_2^{n-1}| \int_{S^{n-1}} |x_1 |^p dS_{B_2^n}(x).
\end{equation}
Using polar coordinates and Fubini's theorem, we get 
\[
\int_{S^{n-1}}|x_1|^p dS_{B_2^n}(x)=(n+p)\int_{B_2^{n-1}}|z_1|^pdz=\frac{2\pi^\frac{n-1}{2}\Gamma\left(\frac{p+1}{2}\right)}{\Gamma\left(\frac{p+n}{2}\right)}.
\]
Thus (\ref{eqball}) becomes
$$
\Gamma\left(\frac{n+1}{2}\right)\Gamma\left(\frac{p+n}{2}\right) \le  \Gamma\left(\frac{n+2}{2}\right) \Gamma\left(\frac{p+n-1}{2}\right).
$$
Using the strict $\log$-convexity of $\Gamma$ function the above is only true if and only if $p \le 2$.
\end{proof}



 \begin{prop} 
 Fix $p>1$.  Then  (\ref{lp-sumconj}) does not hold in the  class of all convex symmetric  bodies in $\RL^n$.
 \end{prop}
\begin{proof} Let us first construct an example in $\RL^2$. We note that
$$
\Big|P_{u^\perp} A\oplus_p P_{u^\perp} \varepsilon^{\frac{1}{p}}[- v,  v] \Big|=2h_A(v)+\frac{2\varepsilon}{p}h_A(v)^{1-p}+o(\varepsilon).
$$
 Thus (\ref{lp-sumconjw}), for $n=2$, would imply
 \begin{equation}\label{2dimp}
 |A|  \le 
h_A(v)^{p} \int\limits_{S^1} |\langle v, x\rangle |^p h_A^{1-p}(x) dS_A(x).
 \end{equation}
Consider $a \in (0, 1),$ let 
$
A=\{(x_1, x_2) \in \RL^2, |x_i|\le 1, |x_1\pm x_2|\le 2-a \}.
$
Then $|A|=4-2a^2$. We  check  (\ref{2dimp}) with $v=e_1$. We first note that $h_A(e_1)=1$. Next we  compute 
$$
f(x)=|\langle e_1, x\rangle |^p h_A^{1-p}(x) S_A(x)
$$
for different normal vectors $x$ of $A$. We first note that $f(\pm e_2)=0$,  $
f(\pm e_1)=2(1-a)$ and 
$ 
f((\pm 1/\sqrt{2}, \pm 1/\sqrt{2}))=a(2-a)^{1-p}.
$
Thus to contradict (\ref{2dimp}) we must select $a$ such that 
$$
4-2a^2 > 4(1-a)+ 4 a(2-a)^{1-p},
$$
or $a< 2-2^{1/p},$ which is possible for every $p>1$.

To build a  counterexample in $\RL^n$ for $n\ge 3$, we use the fact that if $K$ is a convex body in ${\rm span}\{e_1,\dots, e_{n_1}\}$ and $L$ is a convex body in ${\rm span}\{e_{n_1+1}, \dots, e_{n_1+n_2}\}$, then
\begin{equation}\label{lpdirectsum}
|K\oplus_p L|= \frac{\Gamma(\frac{n_1}{q}+1)\Gamma(\frac{n_2}{q}+1)}{\Gamma(\frac{n_1+n_1}{q}+1)} |K||L|=c_{n_1, n_2, q}|K||L|,
\end{equation}
where $1/p+1/q=1$. Let $A_n=B_q^{n-2}\oplus_p A_2$, where $A_2 \subset {\rm span}\{e_1, e_2\}$
is the counterexample created above, and $B^{n-2}_q\subset {\rm span}\{e_3, \dots, e_{n}\}$.
Then (\ref{lpdirectsum}) gives
$$
\frac{|A_n\oplus_p [-v, v]|}{|P_{e_1^\bot}(A_n\oplus_p[-v, v])|_{n-1}} =
\frac{|B_q^{n-2}\oplus_p (A_2\oplus_p[-v, v])|}{|B_q^{n-2}\oplus_p P_{e_1^\bot}( A_2\oplus_p [-v, v])|_{n-1}}< 
\frac{|A_n|}{|P_{e_1^\bot}A_n|_{n-1}}.
$$
\end{proof}

Let us note that the direct interpretation of the volume of $L_p$-zonotopes in terms of determinants is only possible in the case when $p=1$ or $p=2$ \cite{GJ97}, thus it is natural to ask if the determinant inequality that we proved in the case $p=2$ is still true in the case $p\not = 2:$

\begin{ques}\label{question:determ}
Let $p \ge 1$, consider $N\ge N' \ge n$ and  a sequence for vectors $\{u_i\}_{i=1}^N$  in $\RL^n$ is it true that
\begin{align}
 \frac{\sum\limits_{M\subset\{1,\dots,N'\}, |M|=n}|\det(\{u_{m}\}_{m \in M}|^p} {\sum\limits_{M\subset\{1,\dots,N'\},|M|= n-1}|\det(\{P_{e_1^\perp}u_{m}\}_{m \in M})|^p}& +  
\frac{\sum\limits_{L\subset\{N'+1,\dots,N\},|L|= n}|\det(\{u_{l}\}_{l \in L}|^p} {\sum\limits_{L\subset\{N'+1,\dots,N\}, |L|= n-1}|\det(\{P_{e_1^\perp}u_{l}\}_{l \in L}|^p} \nonumber \\
&\le \frac{\sum\limits_{M\subset\{1,\dots,N\},|M|= n}|\det(\{u_{m}\}_{m\in M}|^p}{\sum\limits_{M\subset\{1,\dots,N\}, |M|= n-1}|\det(\{P_{e_1^\perp}(u_{m})\}_{m\in M}|^p}?
\label{eq:determinantsstrongp}
\end{align}
\end{ques}

Notice that the above question also has a negative answer when $p>2$. Let  $N=n=2,$ $N'=1$ and the  matrix 
$$
\begin{pmatrix}
1 & -1& 0\\
1 & 1 & 1 
\end{pmatrix}.
$$
Then  (\ref{eq:determinantsstrongp}) becomes
$$
\frac{2^p}{1+1} \le \frac{2^p+1+1}{1+1+1},
$$
which is false for $p>2$.

\bibliographystyle{plain}
\bibliography{pustak}

\noindent Matthieu Fradelizi \\
Univ Gustave Eiffel, Univ Paris Est Creteil, CNRS, LAMA UMR8050 F-77447 Marne-la-Vall\'ee, France\\
E-mail address: matthieu.fradelizi@univ-eiffel.fr 

\vspace{0.8cm}

\noindent Mokshay~Madiman\\
University of Delaware\\
Department of Mathematical Sciences\\
501 Ewing Hall \\
Newark, DE 19716, USA\\
E-mail address: madiman@udel.edu

\vspace{0.8cm}

\noindent Mathieu Meyer \\
Univ Gustave Eiffel, Univ Paris Est Creteil, CNRS, LAMA UMR8050 F-77447 Marne-la-Vall\'ee, France\\
E-mail address:  mathieu.meyer@univ-eiffel.fr

\vspace{0.8cm}

\noindent Artem Zvavitch \\
Department of Mathematical Sciences \\
Kent State University \\
Kent, OH 44242, USA \\
E-mail address: zvavitch@math.kent.edu

\end{document}